\documentclass[11pt]{amsart}
\title[Corks with large shadow-complexity and exotic $4$-manifolds]
{Corks with large shadow-complexity and exotic $4$-manifolds}

\author[Hironobu Naoe]{Hironobu Naoe}
\address{Tohoku University, Sendai, 980-8578, Japan}
\email{hironobu.naoe.p5@dc.tohoku.ac.jp}
\keywords{4-manifolds, corks, exotic pairs, shadow-complexity. }
\subjclass[2010]{Primary 57R55, 57M50; Secondary 57R65.}

\usepackage{amsfonts,amsmath,amssymb,amscd}
\usepackage{amsthm}
\usepackage{latexsym}
\usepackage{bm}
\usepackage[dvipdfmx]{graphicx}
\usepackage{psfrag}
\usepackage{color}
\usepackage{url}

\theoremstyle{plain}
\newtheorem{theorem}{Theorem}[section]
\newtheorem{lemma}[theorem]{Lemma}
\newtheorem{corollary}[theorem]{Corollary}
\newtheorem{proposition}[theorem]{Proposition}
\newtheorem{question}[theorem]{Question}

\theoremstyle{definition}
\newtheorem{definition}[theorem]{Definition}
\newtheorem{remark}[theorem]{Remark}

\newcommand{\Z}{\mathbb{Z}}
\newcommand{\R}{\mathbb{R}}

\newcommand{\CP}{\mathbb{CP}^2}
\newcommand{\mCP}{\overline{\mathbb{CP}}^2}

\newcommand{\RP}{\mathbb{RP}}

\newcommand{\Int}{\mathrm{Int}}

\newcommand{\gl}{\mathrm{gl}}
\newcommand{\Sing}{\mathrm{Sing}}
\newcommand{\shco}{\mathrm{sc}}
\newcommand{\spshco}{\mathrm{sc}^{\mathrm{sp}}}

\newcommand{\Mobius}{M\"{o}bius }

\definecolor{darkred}{rgb}{.80,.0,.0}

\definecolor{bl}{gray}{0.7}

\newlength{\myheight}
\newlength{\myheighta}
\makeatletter

\long\def\@makecaption#1#2{
  \small
  \vskip\abovecaptionskip
  \sbox\@tempboxa{#1. #2}
  \ifdim \wd\@tempboxa >\hsize
    #1. #2\par
  \else
    \global \@minipagefalse
    \hb@xt@\hsize{\hfil\box\@tempboxa\hfil}
  \fi
  \vskip\belowcaptionskip}

\makeatother
\begin{document}
\begin{abstract}
We construct an infinite family $\{ C_{n,k}\}_{k=1}^{\infty}$ of corks of Mazur type 
satisfying $2n\leq \spshco(C_{n,k})\leq O(n^{3/2})$ for any positive integer $n$. 
Furthermore, using these corks, 
we construct an infinite family $\{(W_{n,k},W'_{n,k})\}_{k=1}^{\infty}$ 
of exotic pairs of $4$-manifolds with boundary 
whose special shadow-complexities satisfy the above inequalities. 
We also discuss exotic pairs with small shadow-complexity.
\end{abstract}

\maketitle
\section{Introduction}
A \textit{complexity} of manifolds is an invariant of manifolds 
which measures how they are simple or complicated. 
In dimension $4$, 
Costantino introduced \textit{shadow-complexity} by 
using Turaev's shadow \cite{Costantino:2006,Turaev:1994}. 
For a compact oriented smooth $4$-manifold $M$ with boundary, 
a \textit{shadow} of $M$ is a simple polyhedron properly embedded in $M$
such that it is locally flat and a strongly deformation retract of $M$. 
The shadow-complexity of $M$, denoted by $\shco(M)$, is defined as the minimal number 
of true vertices of a shadow of $M$. 
We refer the reader to Costantino \cite{Costantino:2006} and Martelli \cite{Martelli:2010} 
for studies of classification of $4$-manifolds according to the (special) shadow-complexity. 
It is known that the shadow-complexity has a crucial relation 
with the Gromov norm $\|\cdot\|$ of the boundary 3-manifolds.
Costantino and Thurston \cite{Costantino_Thurston:2008} showed that 
\[
C_1 \| N\|\leq \shco(N)\leq C_2 \| N\|^2
\]
for any geometric $3$-manifold $N$, 
where $C_1$ and $C_2$ are some universal constants and 
a shadow of $N$ is defined as a shadow of a $4$-manifold bounded by $N$. 
In particular, $N$ has hyperbolic pieces if and only if $\shco(N)\geq1$. 
By definition, we have $\shco(\partial M)\leq \shco(M)$ for any $4$-manifold $M$, 
and hence we can intuitively say that a $4$-manifold bounded by 
a $3$-manifold with high hyperbolicity has a large shadow-complexity. 

Now we focus on an attractive world of dimension $4$. 
Here \textit{exotic} is used in the sense that homeomorphic but not diffeomorphic. 
A $4$-manifold does not always have a unique smooth structure. 
%
It is well-known that any manifold exotic to a given closed and simply-connected $4$-manifold 
can be obtained by removing and regluing a contractible submanifold 
\cite{Matveyev:1996,Curtis_Freedman_Hsiang_Stong:1996}. 
This submanifold is called a \textit{cork}. 
We can assume that a cork is a compact Stein surface due to \cite{Akbulut_Matveyev:1998}. 
A cork plays significant roles in the study of smooth structures of $4$-manifolds. 
By using corks, many exotic pairs were found 
by Akbulut and Yasui \cite{Akbulut_Yasui:2008}.  
They also constructed finitely/infinitely many exotic pairs from a single cork
\cite{Akbulut_Yasui:2013, Akbulut_Yasui:20132, Yasui:2011}. 
Corks are also related to some other topics, 
for example \cite{Karakurt_Oba_Ukida:2017, Yasui:2015}. 

The shadow-complexity works in a rough classification of corks. 
In the previous work \cite{Naoe:2015}, the author found 
infinitely many corks with shadow-complexity $1$ and $2$. 
He also showed in \cite{Naoe:2017} 
that any acyclic $4$-manifold with shadow-complexity zero is diffeomorphic to the $4$-ball. 
In other words, there are no corks having shadow-complexity zero. 
Note that the special shadow-complexity $\spshco(M)$ of a manifold $M$ 
is defined as the minimal number of true vertices of a shadow of $M$ having only disk regions. 
The property that each region is a disk is required to perform hyperbolic Dehn fillings. 
The following is the main result in this paper. 
\begin{theorem}
\label{thm: corks}
For any positive integer $n$, there exists an infinite family $\{ C_{n,k}\}_{k=1}^{\infty}$ 
of corks of Mazur type such that 
\[
2n\leq \spshco(C_{n,k})\leq O(n^{3/2}). 
\]
More precisely, 
$\spshco(C_{n,k})$ is bounded above by a function $D(n)$ given by 
\begin{align*}
D(n)=
\begin{cases}
(n-1)\left\lceil \sqrt{4\pi^2n-1} \right\rceil + n + 
2\left\lceil \sqrt{4\pi^2n-1/4} \right\rceil
& (n:\text{odd}) \\[6pt]
(n-2)\left\lceil \sqrt{4\pi^2n-1} \right\rceil + n + 
4\left\lceil \sqrt{4\pi^2n-1/4} \right\rceil -2
& (n:\text{even}), 
  \end{cases}
\end{align*}
where $\lceil\cdot\rceil$ is the ceiling function. 
\end{theorem}
We will show that each $C_{n,k}$ is a cork, 
the upper bound, the lower bound and the infiniteness 
in Lemmas \ref{lem: cork}, \ref{lem: upper bound}, \ref{lem: lower bound} and \ref{lem: Casson}, 
respectively. 
Theorem \ref{thm: corks} is straightforward from these lemmas. 

The shadow-complexity $\shco(M,M')$ of a pair of $4$-manifolds $M$ and $M'$ is defined to be 
the maximum between $\shco(M)$ and $\shco(M')$. 
The special shadow-complexity $\spshco(M,M')$ is similarly defined. 
In \cite{Costantino:2006}, Costantino asked the following question: 
what is the lowest shadow-complexity/special shadow-complexity of exotic pairs 
of closed/nonclosed $4$-manifolds? 
Recently, Akbulut and Ruberman \cite{Akbulut-Ruberman:2016} 
introduced a new notion called \textit{relatively exotic}. 
A pair $(M,f)$ of a manifold $M$ with boundary and a diffeomorphism $f:\partial M\to \partial M$ 
is called a \textit{relatively exotic manifold} if $f$ extends to a self-homeomorphism of $M$ but 
does not extend to any self-diffeomorphism of $M$. 
In this paper, we conveniently call $M$ a relatively exotic manifold. 
Now we reformulate Costantino's question. 
\begin{question}
\label{question}
${}$
\begin{enumerate}
\item
What is the lowest shadow-complexity of exotic pairs of closed $4$-manifolds? 
\item
What is the lowest special shadow-complexity of exotic pairs of closed $4$-manifolds? 
\item
What is the lowest shadow-complexity of exotic pairs of $4$-manifolds with boundary? 
\item
What is the lowest special shadow-complexity of exotic pairs of $4$-manifolds with boundary? 
\item
What is the lowest shadow-complexity of relatively exotic $4$-manifolds? 
\item
What is the lowest special shadow-complexity of relatively exotic $4$-manifolds? 
\end{enumerate}
\end{question}
Costantino produced upper estimates for Question \ref{question} (1) and (4). 
More precisely, he mentioned in \cite{Costantino:2006} that $\min \shco(M,M')\leq 14$ 
for exotic pairs of closed $4$-manifolds $M$ and $M'$, 
and that $\min \spshco(M,M')\leq3$ for exotic pairs of $4$-manifolds $M$ and $M'$ with boundary. 

We give the complete answer to Question \ref{question} (3). 
\begin{theorem}
\label{thm: lowest shadow-complexity is zero}
The $4$-manifolds $W_1$ and $W_2$ with boundary 
given by the Kirby diagrams in Figure \ref{complexity zero pair} are exotic, 
and the shadow-complexity of this pair is zero. 
\end{theorem}
\begin{figure}
\includegraphics[width=84mm]{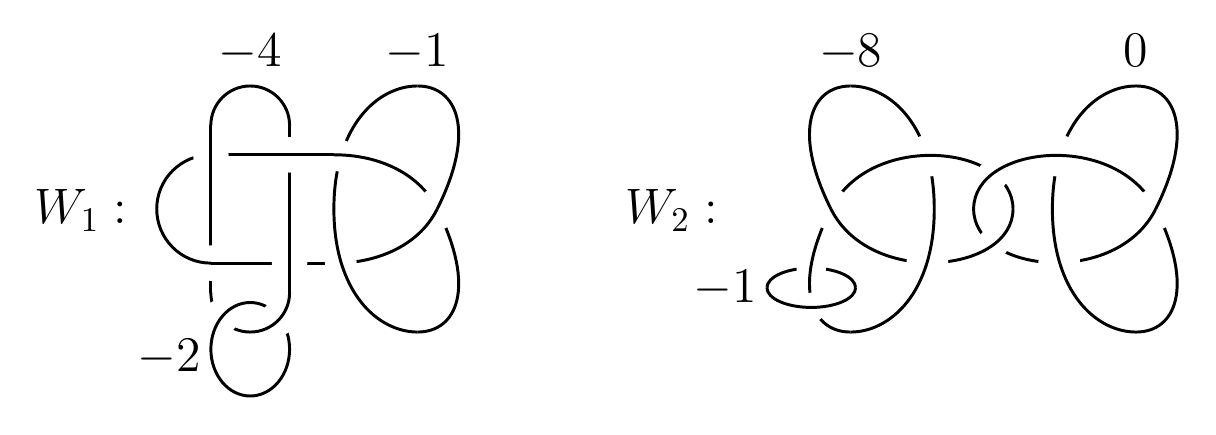}
\caption{The exotic pair having shadow-complexity zero. }
\label{complexity zero pair}
\end{figure}
We next turn to Question \ref{question} (4). 
We will investigate homeomorphism and diffeomorphism types of 
nonclosed $4$-manifolds with shadow-complexity $0$ in Proposition \ref{prop: non zero scsp}. 
This proposition and an easy discussion lead to the following. 
\begin{theorem}
\label{thm: 1 or 2}
The lowest special shadow-complexity of exotic pairs of 4-manifolds with boundary is $1$ or $2$.
\end{theorem}

Of course, a cork is a relatively exotic $4$-manifold 
by definition (see Definition \ref{def:cork}). 
The answers of Question \ref{question} (5) and (6) are at most $1$ 
since there are corks whose shadow-complexity and special shadow-complexity are $1$. 
An \textit{anticork} is also known as a relatively exotic $4$-manifold 
\cite{Akbulut:1991-2, Akbulut:2016}, 
in which one example of anticork was introduced. 
We can check that its shadow-complexity and special shadow-complexity are also $1$ 
by constructing a shadow and computing the hyperbolic volume of the boundary 
with the computer program SnapPy \cite{SnapPy}. 
One might wonder whether relatively exotic $4$-manifolds always have 
(special) shadow-complexity at least $1$. 
We ask the following. 
\begin{question}
Does there exist a relatively exotic $4$-manifold whose boundary has no hyperbolic pieces? 
\end{question}
We note that a $3$-manifold has (special) shadow-complexity zero 
if and only if it is a graph manifold \cite[Proposition 3.31]{Costantino_Thurston:2008}. 

As a consequence of Theorem \ref{thm: corks}, 
we have a result about exotic pairs having large shadow-complexity. 
\begin{corollary}
\label{cor: exotic pair}
For any positive integer $n$, there exists an infinite family $\{ (W_k,W'_k)\}_{k=1}^{\infty}$ 
of exotic pairs of $4$-manifolds with boundary such that 
\[
2n\leq \spshco(C_{n,k})\leq D(n)+2, 
\]
where $D(n)$ is the function in Theorem \ref{thm: corks}. 
\end{corollary}
\subsection*{Acknowledgements}
The author would like to thank his supervisor Masaharu Ishikawa for 
his encouragement and useful comments. 
The author would also like to thank Daniel Ruberman for valuable private communications and 
letting him know examples of corks mentioned in Remark \ref{rmk:primeness of C_{n,k}}. 
He is supported by JSPS KAKENHI Grant Number 17J02915. 

\section{Preliminaries}
Throughout this paper, we assume that any manifold is compact, oriented and smooth, and 
that any map is also smooth. 

\subsection{Corks}
A notion of cork has been known since 1991 by Akbulut \cite{Akbulut:1991}. 
It is shown that a cork appears for any exotic pair of closed and simply-connected 
$4$-manifolds by Matveyev \cite{Matveyev:1996} 
and by Curtis, Freedman, Hsiang and Stong \cite{Curtis_Freedman_Hsiang_Stong:1996}. 
The following is the definition of a cork. 
\begin{definition}
\label{def:cork}
Let $C$ be a contractible $4$-manifold and $f$ be an involution on $\partial C$. 
If $f$ can not extend to any self-diffeomorphism on $C$, then $(C,f)$, or simply $C$, is 
called a \textit{cork}. 
\end{definition}
\begin{remark}
It had been assumed that $f$ can extend a self-homeomorphism on $C$, 
but this condition is always satisfied by being contractible (c.f. \cite{Boyer:1986}). 
A cork is assumed to be a compact Stein surface in several papers, 
for example \cite{Akbulut_Yasui:2008}, though this assumption is excluded in recent papers. 
\end{remark}

A useful criterion to detect a cork by using a Kirby diagram is known as follows. 
\begin{proposition}
[Akbulut and Matveyev \cite{Akbulut_Matveyev:1997}]
\label{prop:cork of Mazur type}
Let $C$ be a $4$-manifold that admits a Kirby diagram consisting of a dotted unknot $K_1$ 
and a $0$-framed unknot $K_2$. Then $C$ is a cork if the following hold: 
\begin{enumerate}
\item
the link $K_1\sqcup K_2$ is symmetric, 
that is, the components $K_1$ and $K_2$ are exchanged by an isotopy in $S^3$; 
\item
the linking number of $K_1$ and $K_2$ is $\pm1$; 
\item
after exchanging the notation of $1$-handle to the ball notation, 
$K_2$ can be placed in the Legendrian position with respect to the standard contact structure 
on $\partial(B^4\cup1\text{-handle})\cong S^1\times S^2$ 
so that its Thurston-Bennequin number is at least $1$. 
\end{enumerate}
\end{proposition}
Such a cork is called a \textit{cork of Mazur type}. 
We refer the reader to Akbulut and Yasui \cite{Akbulut_Yasui:2008} for this proof. 
Moreover, they in \cite{Akbulut_Yasui:2013} showed that 
infinite many mutually exotic $4$-manifolds can be obtained from a cork of Mazur type. 
We note that (2) implies that $C$ is contractible, 
and (1) and (3) imply that there is an involution on $\partial C$ 
which is necessary for $C$ to be a cork. 
The involution is described in the Kirby diagram by the exchange of $\bullet$ and $0$. 
This carries out a surgery $S^1\times B^3$ to $D^2\times S^2$, 
and then does a surgery the other $D^2\times S^2$ to $S^1\times B^3$. 

\subsection{Shadows}
If each point of a compact space $X$ has a neighborhood homeomorphic to one of (i)-(v) 
in Figure \ref{local_model}, then $X$ is called a \textit{simple polyhedron}. 
The set of points of type (ii), (iii) and (v) is called the \textit{singular set} of $X$ 
and denoted by $\Sing(X)$. 
A point of type (iii) is a \textit{true vertex}, 
and each connected component of $\Sing(X)$ 
with true vertices removed is called a \textit{triple line}. 
Each connected component of $X\setminus\Sing(X)$ is called a \textit{region}. 
Hence a region consists of points of type (i) or (iv). 
A region is called an \textit{internal region} if it contains no points of type (iv), 
and a \textit{boundary region} otherwise. 
The \textit{boundary} of $X$, denoted by $\partial X$, 
is defined as the set of points of type (iv) and (v). 
A simple polyhedron is said to be \textit{special} if its regions are open disks. 
\begin{figure}[h]
\includegraphics[width=100mm]{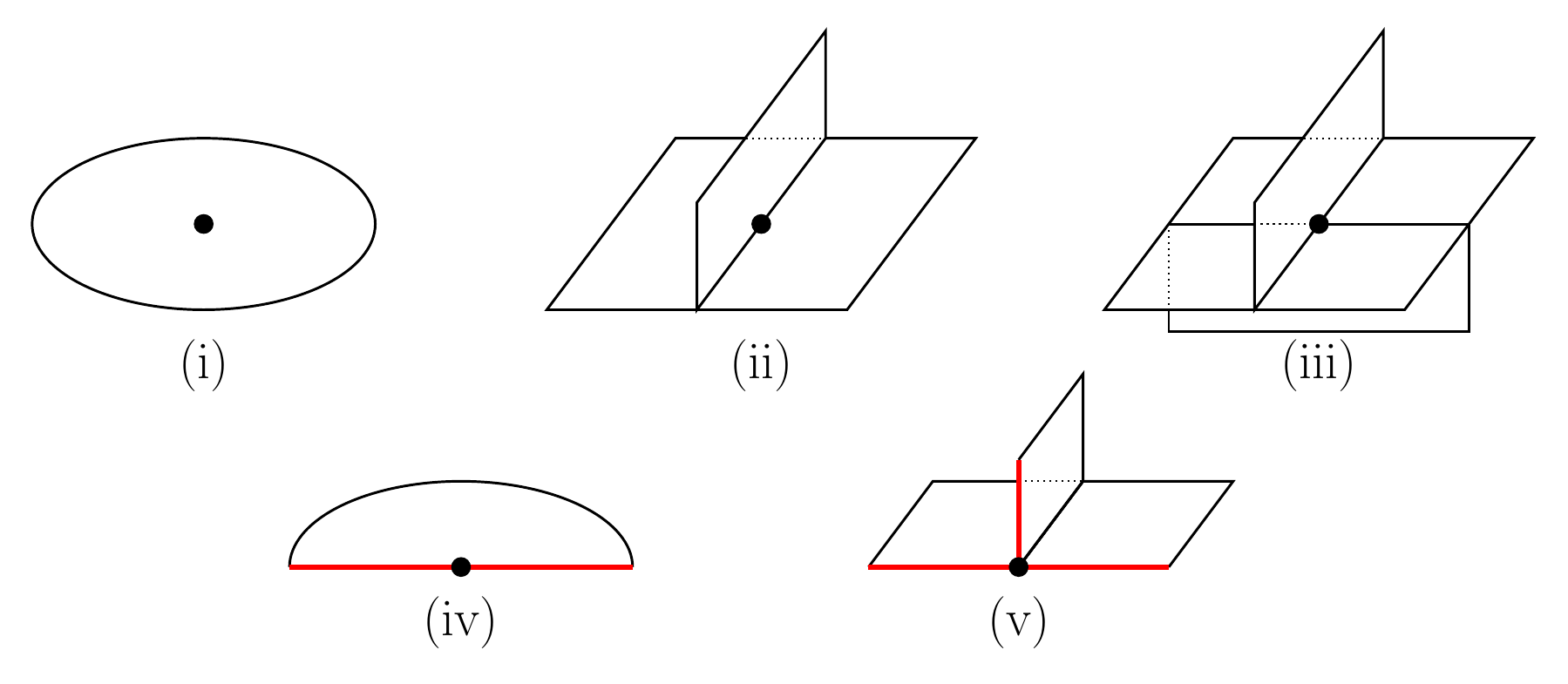}
\caption{The local models of a simple polyhedron.}
\label{local_model}
\end{figure}
\begin{definition}
Let $M$ be a $4$-manifold with boundary, and $X\subset M$ be a simple polyhedron 
that is proper and locally flat in $M$. 
If $X$ is a strongly deformation retract of $M$, 
then the polyhedron $X$ is called a \textit{shadow} of $M$. 
If $X$ is special, then it is called a \textit{special shadow}. 
\end{definition}
We remark that $X$ is said to be locally flat in $M$ 
if there is a local chart $(U,\varphi)$ around each point $x\in X$ 
such that $\varphi(U\cap X)$ is contained in $\R^3\subset\R^4=\varphi(U)$.
It is easy to see that any handlebody consisting of $0$-, $1$- and $2$-handles 
admits a shadow \cite{Turaev:1994, Costantino:2005}. 

\subsection{Gleams}
For any simple polyhedron $X$, 
one can define the \textit{$\Z_2$-gleam} on each internal region. 
Let $R$ be an internal region, and $i:F\to X$ be a continuous map 
extended from the inclusion of $R$, 
where $F$ is a compact surface whose interior is homeomorphic to $R$. 
Note that the restriction $i|_{\Int(F)}$ coincides with the inclusion of $R$, 
and that $i(\partial F)\subset\Sing(X)$. 
We now see that there exists a local homeomorphism $\tilde{i}:\tilde{F}\to X$ 
such that its image is a neighborhood of $i(F)$ in $X$, 
where $\tilde{F}$ is a simple polyhedron obtained from $F$ by 
attaching an annulus or a \Mobius strip along its core circle to each boundary component of $F$. 
Note that $\tilde{F}$ is determined up to homeomorphism from the topology of $X$. 
Here the $\Z_2$-gleam $\gl_2(R)$ of $R$ is defined to be $0$ 
if the number of the attached \Mobius strips is even, and $1$ otherwise. 
\begin{definition}
A simple polyhedron $X$ is called a \textit{shadowed polyhedron} 
if each internal region $R$ is equipped with a half integer $\gl(R)$ such that 
$\gl(R)-\frac{1}{2}\gl_2(R)$ is an integer. 
The half integer $\gl(R)$ is called a \textit{gleam} on $R$. 
\end{definition}
\begin{theorem}
[Turaev, \cite{Turaev:1994}]
\begin{enumerate}
\item 
There exists a canonical way to construct a $4$-manifold $M_X$
from a given shadowed polyhedron $X$ such that $X$ is a shadow of $M_X$. 
This construction provides a smooth structure on $M_X$ uniquely. 
\item
Let $M$ be a $4$-manifold admitting a shadow $X$. 
Then there exist gleams on internal regions of $X$ 
such that $M$ is diffeomorphic to the $4$-manifold constructed from the shadowed polyhedron 
according to the way of (1). 
\end{enumerate}
\end{theorem}
The construction in (1) is called Turaev's reconstruction. 
A gleam plays a role as a framing coefficient to attach a $2$-handle 
in the original proof of Turaev's reconstruction. 
It is also regarded as a generalized Euler number of an embedded surface in a $4$-manifold. 
In the case where a $4$-manifold is a $D^2$-bundle over a surface $F$, 
the $4$-manifold has a shadow $F$, and the Euler number of $F$ coincides with 
the gleam coming from the above theorem. 

Now we introduce a way to calculate gleams from link projection. 
Let $H$ be a $4$-dimensional handlebody consisting of $0$-handles and $1$-handles, 
and $M$ be a $4$-manifold obtained from $H$ 
by attaching $2$-handles along a framed link $L=L_1\cup\cdots\cup L_n$ in $\partial H$. 
Let $Y$ be a shadow of $H$ such that the gleams of regions of $Y$ are all $0$. 
We then project $L$ onto $Y$ in a regular position. 
Here $\pi$ denotes the projection. 
By attaching a $2$-disk $D_i$ to $\pi(L_i)$ for each $i\in\{1,\ldots,n\}$, 
we obtain a new simple polyhedron $X$, which is a shadow of $M$. 
The regions of this polyhedron other than the $2$-disks $D_1,\ldots,D_n$ are subsets of $Y$. 
\begin{itemize}
\item 
The gleam of $D_i$ coincides with the framing of $L_i$ with respect to the one induced from $Y$. 
More precisely, it is given as follows. 
Let $L'_i$ denote the framing of $L_i$. 
We may assume that 
the image of $L'_i$ under $\pi$ is parallel to $\pi(L_i)$ on $Y$ 
except for an arc $\alpha_i\subset L'_i$, 
and $\alpha_i$ is sent so that its image has normal crossings with $\pi(L_i)$. 
We assign an over/under information at each crossing point. 
Each has a sign canonically. 
Then the gleam of $D_i$ is given as the half of the total number of the positive crossings 
minus the total number of the negative ones. 
\item
Let $R$ be an internal region contained in $Y\subset X$. 
Then $R$ might be adjacent to some crossing points of the link projection, or 
intersection points of $\Sing(Y)$ and the link projection as shown in Figure \ref{local_gleam}. 
Around these points, 
we provide local contributions to the gleam on $R$ as shown in Figure \ref{local_gleam}, 
and the gleam $\gl(R)$ is given as the sum of them. 
If $R$ is not adjacent to a point as above, then $\gl(R)$ is zero. 
\end{itemize}
\begin{figure}
\includegraphics[width=83mm]{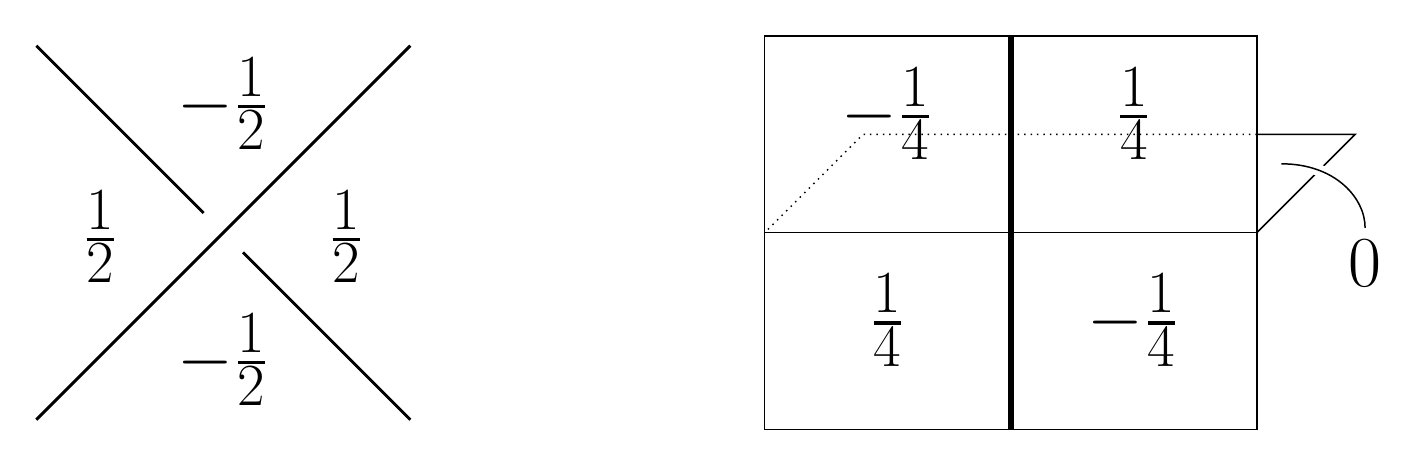}
\caption{The local contribution to gleam. 
The left part indicates a crossing point of the link projection, 
and the right part does an intersection point of $\Sing(Y)$ and the link projection. } 
\label{local_gleam}
\end{figure}

\subsection{Shadow-complexity}
The shadow-complexity was first introduced by Costantino in \cite{Costantino:2006}. 
He studied the shadow-complexity of closed $4$-manifolds. 
Note that a shadow of a closed $4$-manifold $M$ is defined 
as a shadow of 
the union of $0$-, $1$- and $2$-handles in a handle decomposition of $M$. 
Recall that a shadow of a $3$-manifold $N$ is a shadow of a $4$-manifold bounded by $N$. 
\begin{definition}
Let $M$ be a $3$- or $4$-manifold having a shadow. 
Then the (\textit{special}) \textit{shadow-complexity} of $M$ is the minimal number 
of true vertices of a (special) shadow of $M$. 
Let $\shco(M)$ and $\spshco(M)$ denote the shadow-complexity of $M$ 
and the special one, respectively. 
\end{definition}
Costantino \cite{Costantino:2006} 
showed that the special shadow-complexity of a closed $4$-manifold is $0$ 
if and only if it is diffeomorphic to one of 
$S^4$, $\CP$, $\mCP$, $S^2\times S^2$, $\CP\#\CP$, $\mCP\#\mCP$ or $\CP\#\mCP$. 
We will investigate $3$-manifolds having special shadow-complexity $0$ 
in the proof of Proposition \ref{prop: non zero scsp} 
and show that they are Seifert manifolds or connected-sums of lens spaces. 
Note that it is known that a $3$-manifold has shadow-complexity $0$ 
if and only if it is a graph manifold by Costantino and Thurston \cite{Costantino_Thurston:2008}. 
In the $4$-dimensional case, 
Martelli \cite{Martelli:2010} showed an analogue to the above result: 
a $4$-manifold $M$ has shadow-complexity $0$ if and only if 
$M$ is diffeomorphic to $M'\#_k\CP$ or $M'\#_k\mCP$ for some $k\geq0$, 
where $M'$ is a ``$4$-dimensional graph manifold'' discussed in his paper. 

As mentioned in the introduction, Costantino and Thurston 
\cite[Theorems 3.37 and 5.5]{Costantino_Thurston:2008} discovered 
a splendid relation between shadow-complexity and geometry of $3$-manifolds: 
There is a constant $C$ such that 
\[
\frac{v_{\text{tet}}}{2v_{\text{oct}}}\| N\|\leq \shco(N)\leq C\| N\|^2
\]
for any geometric $3$-manifold $N$. 
Moreover, the first inequality holds for every $3$-manifold. 
Here $v_{\text{tet}}=1.01...$ and $v_{\text{oct}}=3.66...$ 
are the volume of regular ideal tetrahedron and octahedron, respectively, 
and $\| N\|$ denotes the Gromov norm of $N$. 
Costantino and Thurston also showed that the $3$-manifold with boundary corresponding to 
a neighborhood of the singular set of a shadow admits a hyperbolic structure. 
By using hyperbolic Dehn filling, Ishikawa and Koda \cite[Theorem 6.2]{Ishikawa_Koda:2017} 
showed the following. 
\begin{theorem}
[Ishikawa and Koda, \cite{Ishikawa_Koda:2017}]
\label{thm:Ishikawa-Koda}
Let $X$ be a special shadow of a $3$-manifold $N$, 
and $X$ has $V$ true vertices. 
For a region $R$, let $v(R)$ denote the number of true vertices 
adjacent to $R$ counted with multiplicity. 
If $\sqrt{4\gl(R)^2+v(R)^2}>2\pi\sqrt{2V}$ for each region $R$, then $\spshco(N)=V$. 
\end{theorem}
Note that $\sqrt{4\gl(R)^2+v(R)^2}$ indicates the slope length of 
the Dehn filling corresponding to $R$.
\begin{remark}
In \cite[Theorem 6.2]{Ishikawa_Koda:2017}, 
a shadow $X$ is assumed to be equipped with a \textit{branching}, 
which is a choice of a suitable orientation of each region. 
Their proof works even if $X$ is not given a branching. 
Hence the theorem still holds without assuming a branching. 
\end{remark}

\section{Cork $C_{n,k}$ and lemmas}
\begin{figure}[h]
\includegraphics[width=90mm]{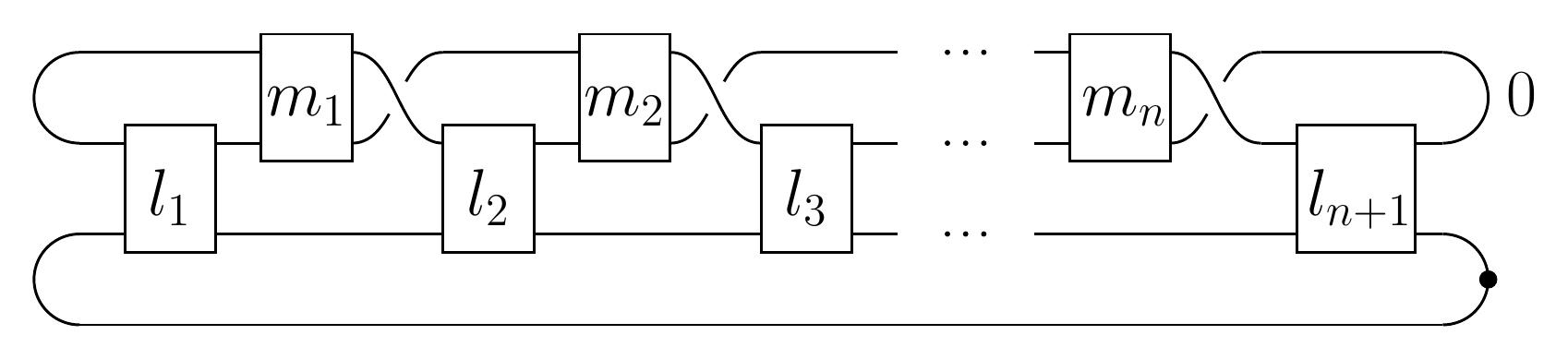}
\caption{The Kirby diagram of $C_{n,k}$}
\label{Kirby diag of Cn}
\end{figure}
We prove Theorem \ref{thm: corks} in this section. 
Let $n$ be a positive integer and $k$ be a nonnegative integer. 
For $i\in\{1,\ldots,n\}$ and $j\in\{1,\ldots,n+1\}$, 
define integers $m_i$ and $l_j$ to be
\begin{align*}
&m_i=
\begin{cases}
-\left\lceil \frac{1}{2}+\sqrt{4\pi^2n-1}\right\rceil -k& (i=1) \\[6pt]
-\left\lceil \frac{1}{2}+\sqrt{4\pi^2n-1}\right\rceil & (\mathrm{otherwise}),
\end{cases}
\intertext{and}
&l_j=
\begin{cases}
\left\lceil\sqrt{4\pi^2n-1/4}\right\rceil & (j=1) \\[6pt]
2\left\lceil\sqrt{4\pi^2n-1/4}\right\rceil-1 & (n:\mathrm{even\ and\ }j=n) \\[6pt]
\left\lceil\sqrt{4\pi^2n-1/4}\right\rceil +1 & (n:\mathrm{odd\ and\ }j=n+1) \\[6pt]
\left\lceil\sqrt{4\pi^2n-1/4}\right\rceil & (n:\mathrm{even\ and\ }j=n+1) \\[6pt]
\left\lceil\sqrt{4\pi^2n-1}\right\rceil & (\mathrm{otherwise}),
\end{cases}
\end{align*}
where $\lceil\cdot\rceil$ is the ceiling function, that is, 
$\lceil x\rceil = \min\{n\in\Z\mid x\leq n \}$ for a real number $x$. 
Let $C_{n,k}$ be a $4$-manifold given 
by the Kirby diagram shown in Figure \ref{Kirby diag of Cn}. 
We adopt a convention that 
a box with some integer $m$ in a link diagram represents $m$ full twists. 

We first show the following. 
\begin{figure}
\includegraphics[width=90mm]{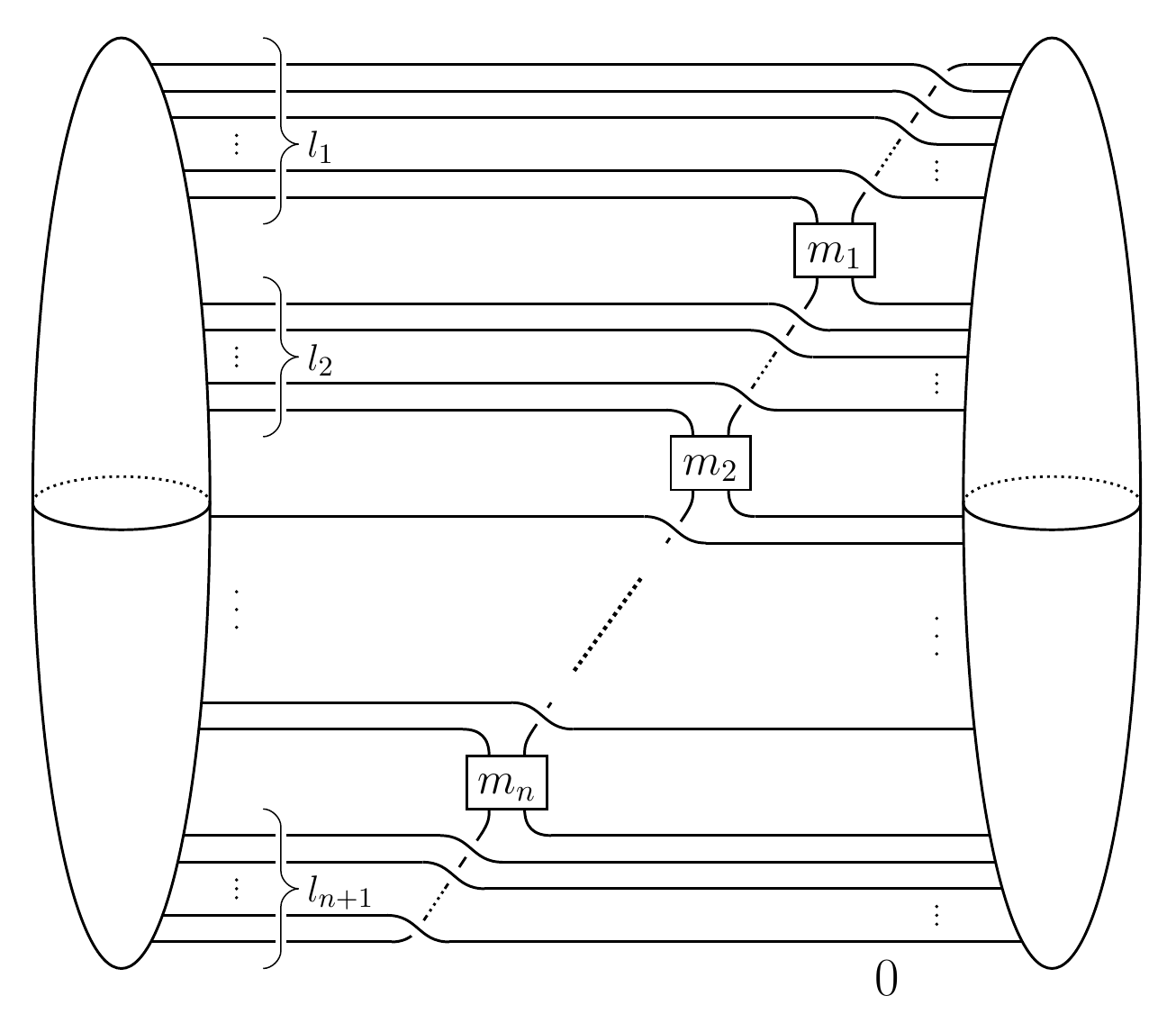}
\caption{The Kirby diagram of $C_{n,k}$ (ball-notation).}
\label{Stein position}
\end{figure}

\begin{lemma}
\label{lem: cork}
The manifold $C_{n,k}$ is a cork of Mazur type. 
\end{lemma}
\begin{proof}
We only have to check the conditions (1)-(3) in Proposition \ref{prop:cork of Mazur type}. 
\begin{enumerate}
 \item 
The link pictured in Figure \ref{Kirby diag of Cn} is a $2$-bridge link.
Hence the components can exchange their positions by an isotopy of $S^3$. 
\item
The linking number of them is 
$\displaystyle \sum_{j=1}^{n+1}(-1)^jl_j=\pm1$ (for some choice of orientations). 
\item
Figure \ref{Stein position} shows a Kirby diagram of $C_{n,k}$ 
after changing the notation of the $1$-handle to the ball-notation. 
The Thurston-Bennequin number of the attaching circle of the $2$-handle is
\[
\sum_{i=1}^{n}2m_i + \sum_{j=1}^{n+1}(l_i-1) - \sum_{i=1}^{n}(2m_i-1) 
= \sum_{j=1}^{n+1}l_i-1\geq1. 
\]
\end{enumerate}
The proof is completed. 
\end{proof}

\begin{figure}
\includegraphics[width=79mm]{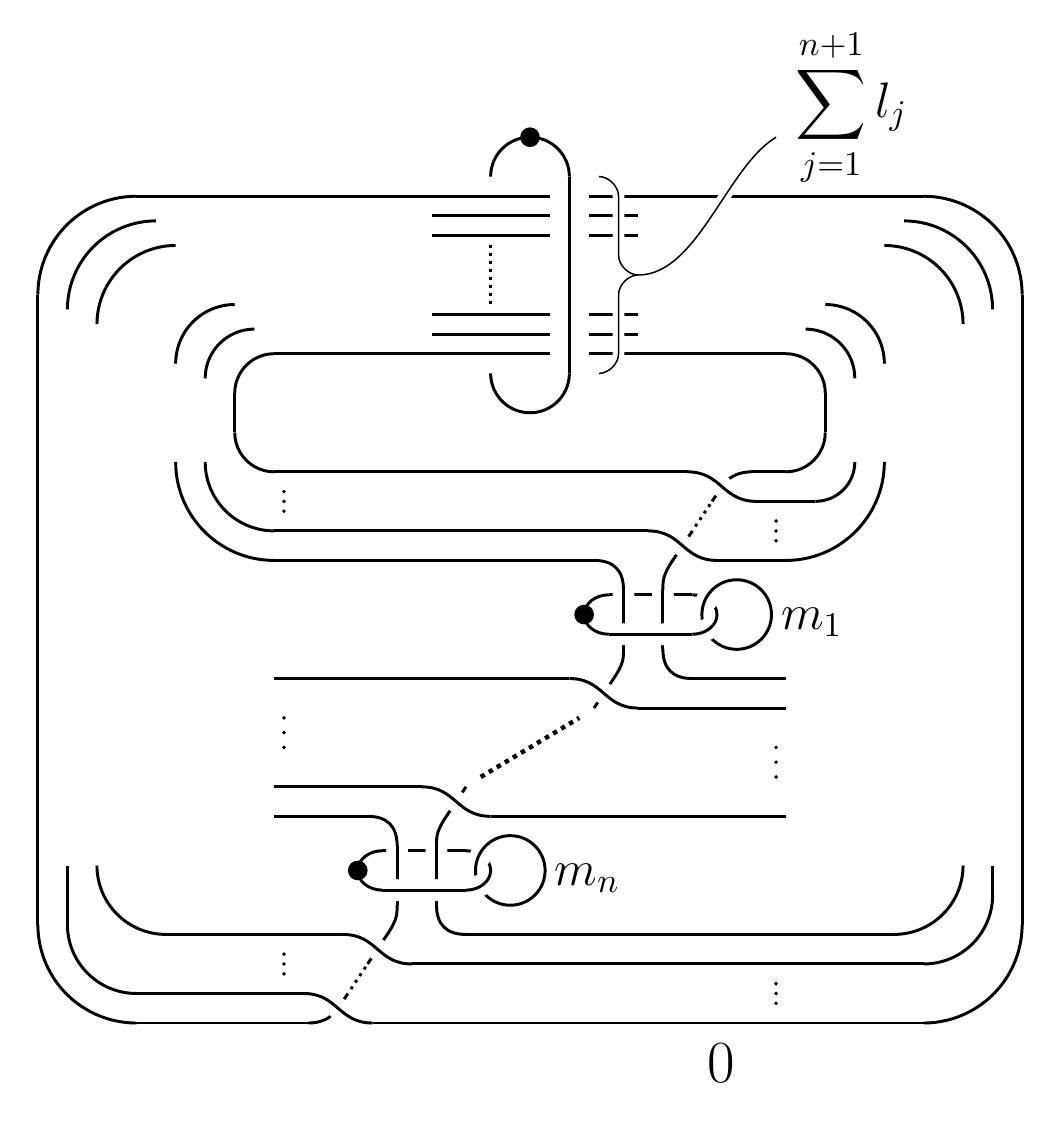}
\caption{The Kirby diagram of $C_{n,k}$.}
\label{to_shadow1}
\includegraphics[width=80mm]{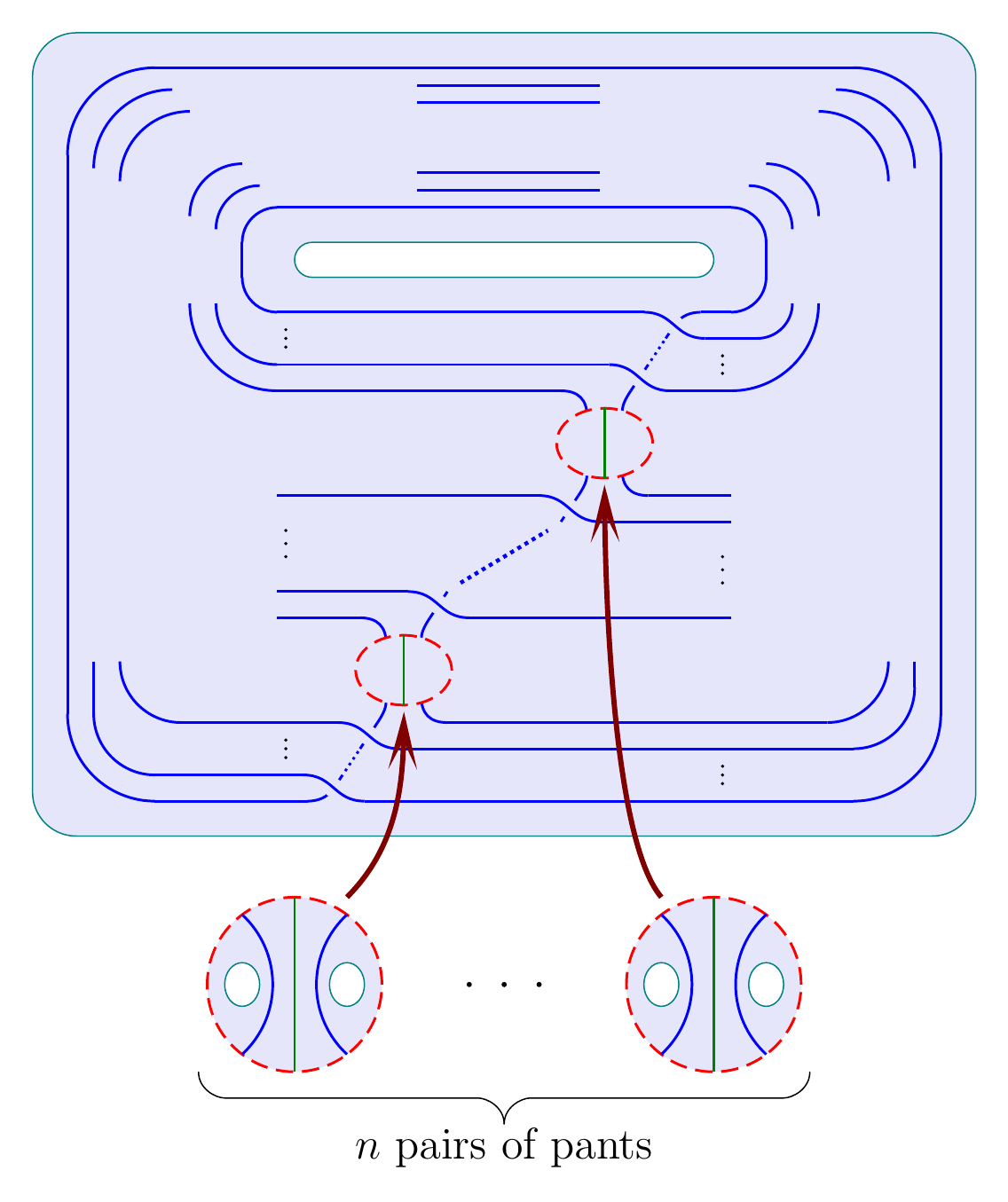}
\caption{The simple polyhedron $P$ 
and curves that are the images of the attaching circles of the $2$-handles.}
\label{to_shadow2}
\end{figure}
We then give an upper bound of $\spshco(C_{n,k})$ by constructing a shadow of $C_{n,k}$. 
\begin{lemma}
\label{lem: upper bound}
$\displaystyle \spshco(C_{n,k})\leq \sum_{j=1}^{n+1}l_j+n-3$.
\end{lemma}
\begin{proof}
Figure \ref{to_shadow1} shows a Kirby diagram of $C_{n,k}$ 
obtained from one shown in Figure \ref{Stein position} by 
changing the notation of the $1$-handle and adding $n$ cancelling pairs. 
Then we see that $C_{n,k}$ admits a handle decomposition with 
one $0$-handle, $n+1$ $1$-handles and $n+1$ $2$-handles. 

We now construct a shadow of $C_{n,k}$. 
Let $P$ be a simple polyhedron obtained from an annulus by 
attaching $n$ pairs of pants to disjoint $n$ simple closed curves 
as shown in Figure \ref{to_shadow2}. 
Then $P$ is a shadow of 
the union of a $0$-handle and $1$-handles 
with respect to the above handle decomposition of $C_{n,k}$. 
The attaching circles of the $2$-handles are projected onto $P$ 
as shown in Figure \ref{to_shadow2}. 
We attach $n+1$ $2$-disks along these curves.  
The obtained polyhedron is simple. We denote this polyhedron by $P'$. 
This polyhedron $P'$ is a shadow of $C_{n,k}$. 
By collapsing along each boundary region of $P'$, 
we obtain a special polyhedron $P''$ which is also a shadow of $C_{n,k}$. 

There are $\displaystyle \sum_{j=1}^{n+1}(l_j-1)+6n$ true vertices in $P'$. 
Two true vertices are adjacent to the boundary regions on the annulus, 
and $4$ true vertices are adjacent to boundary regions on each pair of pants. 
Hence $4n+2$ true vertices in total vanishes by the collapsing. 
Therefore $P''$ has $\displaystyle \sum_{j=1}^{n+1}l_j+n-3$ true vertices. 
The shadow-complexity of $C_{n,k}$ is bounded above by this number. 
\end{proof}
The next lemma gives a lower estimate of $\spshco(C_{n,k})$. 

\begin{lemma}
\label{lem: lower bound}
$\spshco(\partial C_{n,k})=2n$.
\end{lemma}
\begin{proof}
\begin{figure}
\includegraphics[width=120mm]{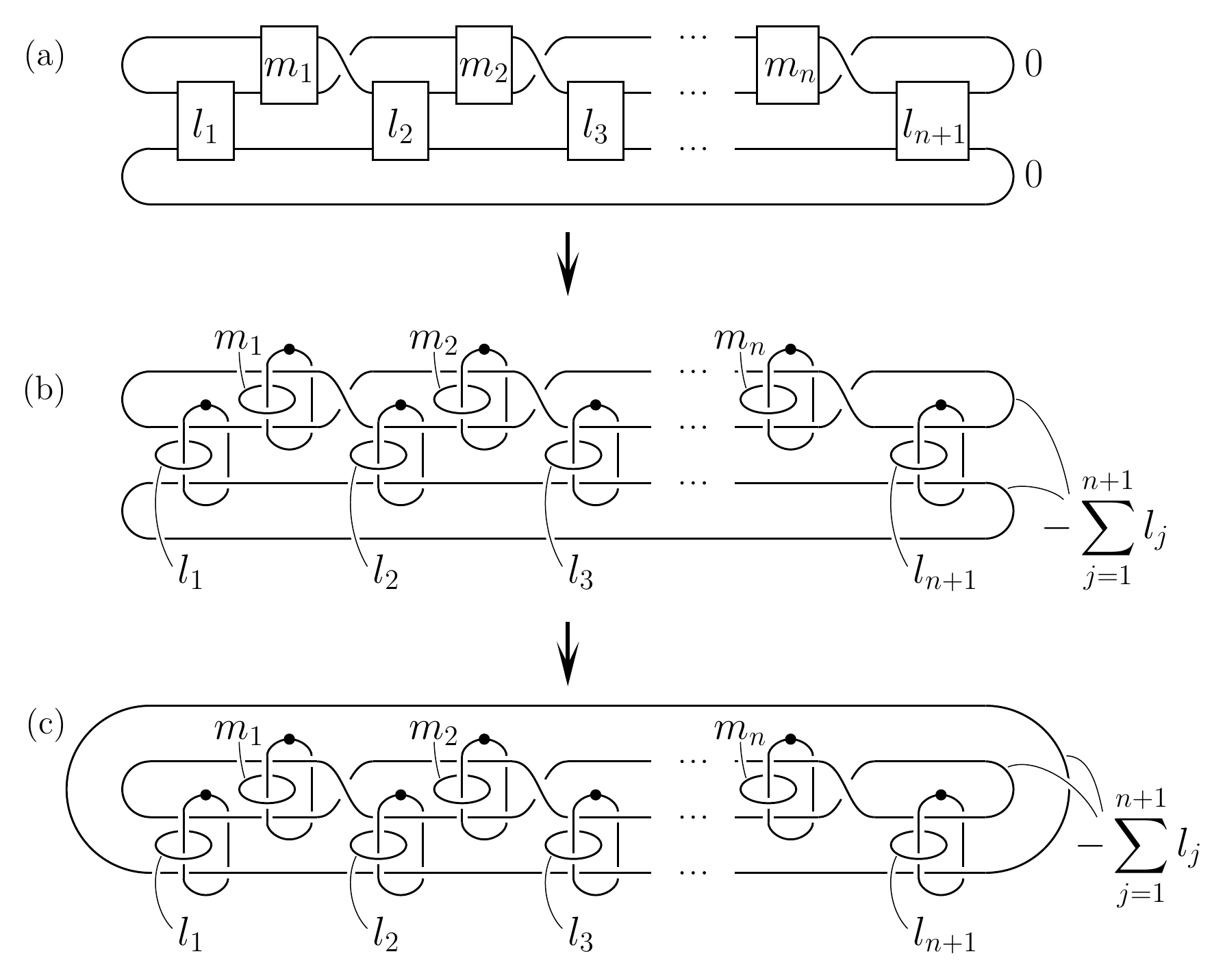}
\caption{The Kirby diagram of $M_{n,k}$.}
\label{2-bridge2}
\end{figure}
Let us consider a new $4$-manifold $M_{n,k}$ 
represented by the diagram shown in Figure \ref{2-bridge2}-(a). 
Note that the boundaries of $M_{n,k}$ and $C_{n,k}$ are diffeomorphic, 
so we compute $\spshco(\partial C_{n,k})$ by considering a shadow of $M_{n,k}$. 
By adding $2n+1$ cancelling pairs of $1$- and $2$-handles, 
$M_{n,k}$ admits the diagram shown in Figure \ref{2-bridge2}-(b). 
We get the diagram shown in Figure \ref{2-bridge2}-(c) by an isotopy. 

We then construct a shadow of $M_{n,k}$ from \ref{2-bridge2}-(c). 
Let $Q$ be a simple polyhedron obtained from an $(n+1)$-holed disk by 
attaching $n$ pairs of pants to disjoint $n$ simple closed curves 
as shown in Figure \ref{shadow of Mnk}. 
\begin{figure}
\includegraphics[width=120mm]{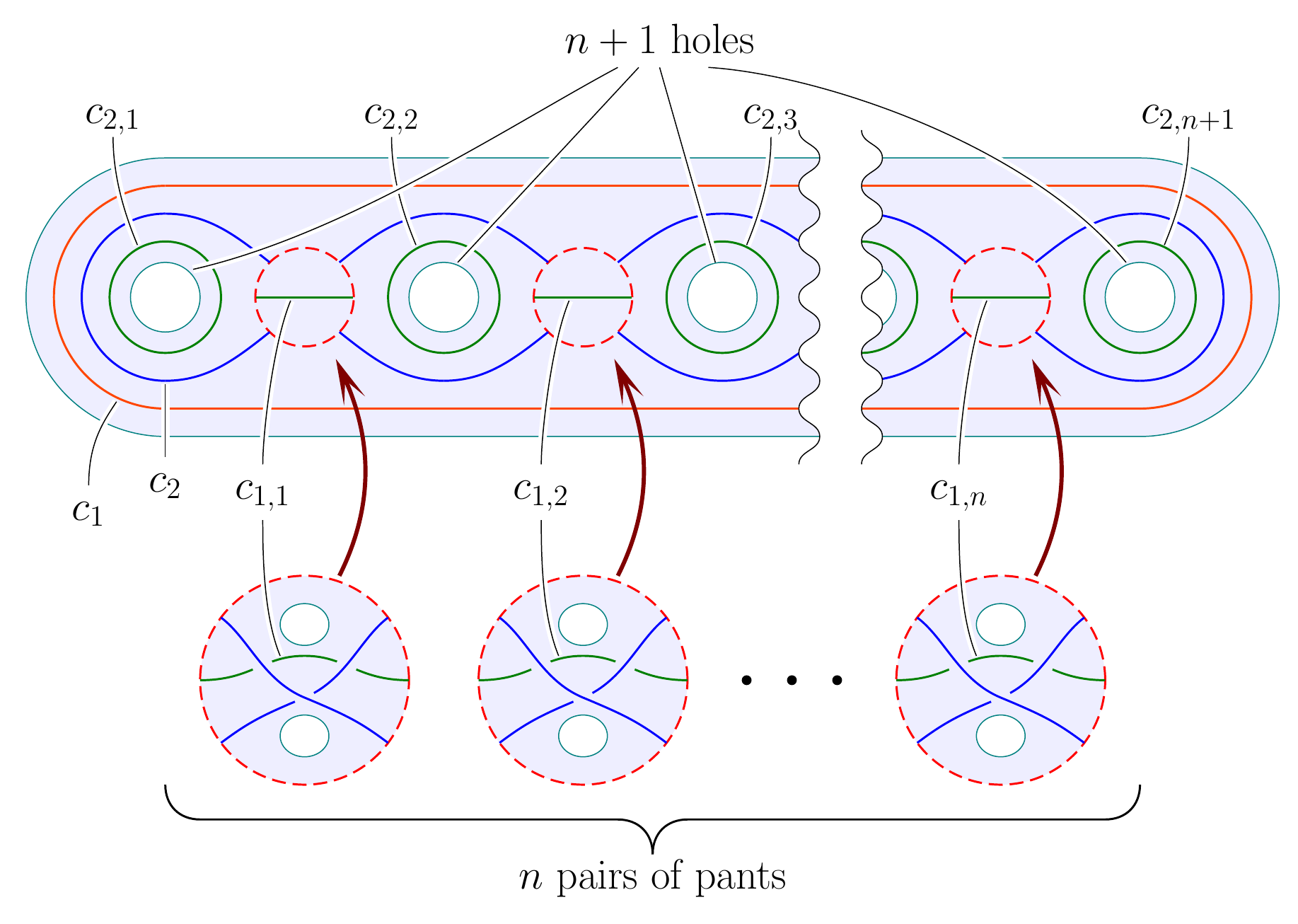}
\caption{The shadow $Q'$ of $M_{n,k}$ and curves on $Q'$ 
corresponding to the $2$-handles of $M_{n,k}$.}
\label{shadow of Mnk}
\includegraphics[width=120mm]{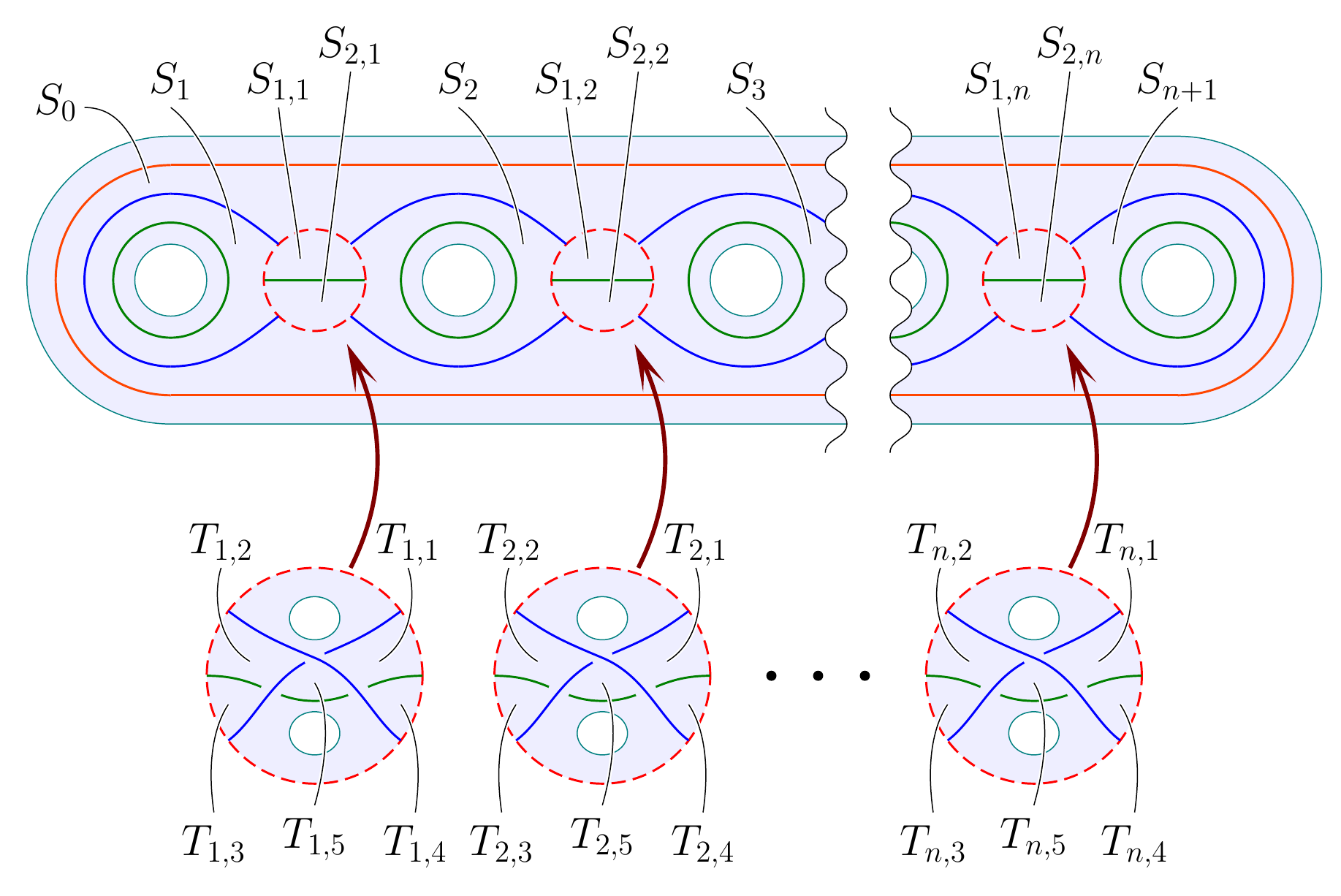}
\caption{The internal regions on $Q'$.}
\label{regions of Q'}
\end{figure}
The union of a $0$-handle and $1$-handles 
retracts onto $Q$. 
Then the attaching circles of the $2$-handles are projected onto $Q$ 
as shown in Figure \ref{shadow of Mnk}. 
We denote them as follows: 
\begin{itemize}
 \item 
let $c_1$ and $c_2$ be the images of the two attaching circles 
whose framing coefficients are $\displaystyle -\sum_{j=1}^{n+1}l_j$; 
 \item 
let $c_{1,i}$ be the image of the attaching circle whose 
framing coefficient is $m_i$ for $i\in\{1,\ldots,n\}$;
 \item 
let $c_{2,j}$ be the image of the attaching circle whose framing coefficient is $l_j$ 
for $j\in\{1,\ldots,n+1\}$. 
\end{itemize} 
The number of the curves above is $2n+3$ in total. 
Then we attach $2n+3$ $2$-disks to these curves. 
Let $D_1$, $D_2$, $D_{1,i}$ and $D_{2,j}$ be the disks attached to 
the curves $c_1$, $c_2$, $c_{1,i}$ and $c_{2,j}$, respectively. 
Let $Q'$ denote the obtained polyhedron. 
It has $9n$ true vertices. 

The polyhedron $Q'$ has $10n+5$ internal regions and $3n+2$ boundary regions. 
As shown in Figure \ref{regions of Q'}, 
\begin{itemize}
\item 
the $(n+1)$-holed disk is divided into $3n+2$ internal regions 
$S_0$, $S_1, \ldots, S_{n+1}$, $S_{1,1},\ldots,S_{1,n},\ S_{2,1},\ldots,S_{2,n}$ 
and $n+2$ boundary regions, and 
\item 
the pair of pants which $c_{1,i}$ passes through is divided into $5$ internal regions 
$T_{i,1},\ldots,T_{i,5}$ and $2$ boundary regions. 
\end{itemize}
We then compute their gleams. 
The gleams on $D_1$, $D_{1,i}$ and $D_{2,j}$ coincide with the framing coefficients 
corresponding to the $2$-handles, but the gleam on $D_2$ does not 
since $c_2$ has a self-crossing point on each pair of pants. 
Thus we get Table \ref{table:gl of D}. 
\begin{table}[h]
\centering
\begin{tabular}[c]{|c||c|c|c|c|}
\hline
${}$
& $D_1$
& $D_2$
& $D_{1,i}$
& $D_{2,j}$\\
\hline \hline
$\gl$
& $\displaystyle -\sum_{j=1}^{n+1}l_j$ 
& $\displaystyle -\sum_{j=1}^{n+1}l_j+n$ 
& $m_i$ 
& $l_j$\\
\hline
\end{tabular}
\caption{The gleams of $D_{\ast}$.}
\label{table:gl of D}
\end{table}
\newline 
According to the local contribution to gleams shown in Figure \ref{local_gleam}, 
we obtain the gleams of the other internal regions as shown in Table \ref{table:gl of S,T}. 
\begin{table}[h]
\centering
\begin{tabular}[c]{|c||c|c|c|c|c|c|c|c|c|}
\hline
${}$
& $S_1$
& $S_j$
& $S_{1,i}$
& $S_{2,i}$
& $T_{i,1}$
& $T_{i,2}$
& $T_{i,3}$
& $T_{i,4}$
& $T_{i,5}$\\
\hline \hline
\parbox[c][\myheighta][c]{0cm}{}\vbox{\hbox{$\gl$}} 
& $0$ 
& $0$ 
& $0$ 
& $0$ 
& $-\frac{1}{2}$ 
& $-\frac{1}{2}$ 
& $0$ 
& $0$ 
& $\frac{1}{2}$\\ 
\hline
\end{tabular}
\caption{The gleams of $S_{\ast}$ and $T_{\ast}$.}
\label{table:gl of S,T}
\end{table}

By collapsing along boundary regions of $Q'$, it becomes a special polyhedron, 
which we denote by $Q''$. 
This polyhedron is a special shadow of $M_{n,k}$, 
which has $2n+3$ internal regions, no boundary regions and $2n$ true vertices. 
Let $R_1$, $R_2$, $R_{1,i}$ and $R_{2,j}$ be the regions of $Q''$ 
containing $D_1$, $D_{1,i}$ and $D_{2,j}$, respectively. 
More precisely define 
\begin{align*}
R_1\ 
&= D_1\cup S_0\cup \left( \bigcup_{i=1}^{n} S_{1,i}\cup S_{2,i} \right),\\
R_2\ 
&= D_2\cup \left( \bigcup_{i=1}^{n} T_{i,1}\cup T_{i,2}\cup T_{i,3}\cup T_{i,4} \right),\\
R_{1,i} 
&= D_{1,i}\cup T_{i,5}\text{, and}\\
R_{2,j} 
&= D_{2,j}\cup S_{j}.
\end{align*}
Taking the sums of the gleams of the subregions, 
we obtain Table \ref{table:gl of R}. 
\begin{table}[h]
\begin{center}
\begin{tabular}[c]{|c||c|c|c|c|}
\hline
${}$
& $R_1$
& $R_2$
& $R_{1,i}$
& $R_{2,j}$\\
\hline \hline
$\gl$
& $\displaystyle -\sum_{j=1}^{n+1}l_j$ 
& $\displaystyle -\sum_{j=1}^{n+1}l_j$ 
& $\displaystyle m_i+\frac{1}{2}$ 
& $l_j$\\
\hline
\end{tabular}
\caption{The gleams of $R_{\ast}$.}
\label{table:gl of R}
\end{center}
\end{table}

By using Theorem \ref{thm:Ishikawa-Koda}, 
we can determine the special shadow-complexity of $\partial M_{n,k}$. 
Recall that $v(R)$ is the number of true vertices adjacent to a region $R$ with multiplicity. 
It is easily seen that 
\begin{align*}
v(R_1)\ &= 4n,\\
v(R_2)\ &= 4n,\\
v(R_{1,i}) &= 2,\\
v(R_{2,j}) &= \begin{cases}
1& (j=1,n+1) \\[6pt]
2 & (\mathrm{otherwise}).
\end{cases}
\end{align*}
An easy computation shows that 
$\sqrt{4\gl(R_{\ast})^2+v(R_{\ast})^2}>2\pi\sqrt{2\cdot2n}$ for every region $R_{\ast}$ of $Q''$. 
Hence 
$\spshco(\partial M_{n,k})=\spshco(\partial C_{n,k})=2n$ by Theorem \ref{thm:Ishikawa-Koda}. 
\end{proof}
\begin{remark}
Our special polyhedron $Q''$ does not admit a branching. 
\end{remark}
We finally show that there are infinite $C_{n,k}$'s in the following. 
\begin{lemma}
\label{lem: Casson}
The two manifolds $\partial C_{n,k}$ and $\partial C_{n,k'}$ are not homeomorphic 
unless $k=k'$. 
\end{lemma}
\begin{proof}
Assume that $k'-k=d>0$. 
Set $m_1=-\left\lceil \frac{1}{2}+\sqrt{4\pi^2n-1}\right\rceil -k$ and 
$m'_1=-\left\lceil \frac{1}{2}+\sqrt{4\pi^2n-1}\right\rceil -k'$.
\begin{figure}\
\includegraphics[width=100mm]{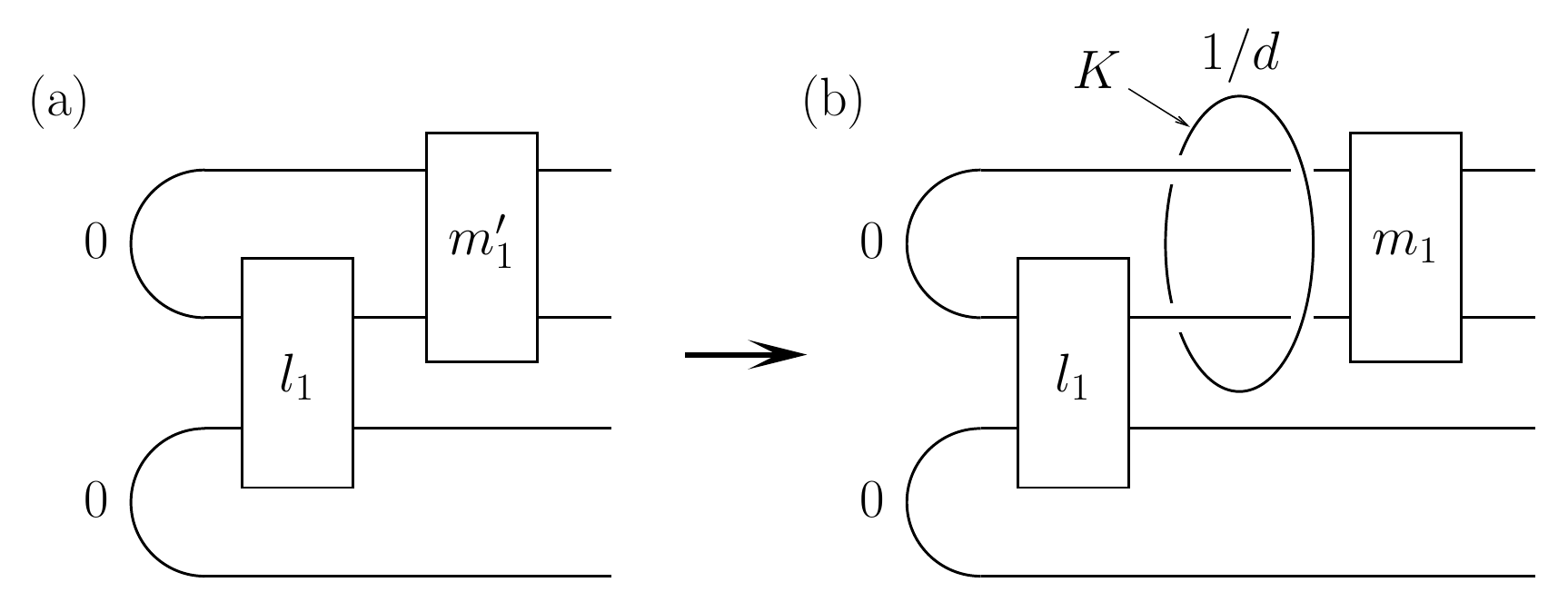}
\caption{A part of the surgery diagram of $\partial C_{n,k'}$. }
\label{Casson1}
\includegraphics[width=70mm]{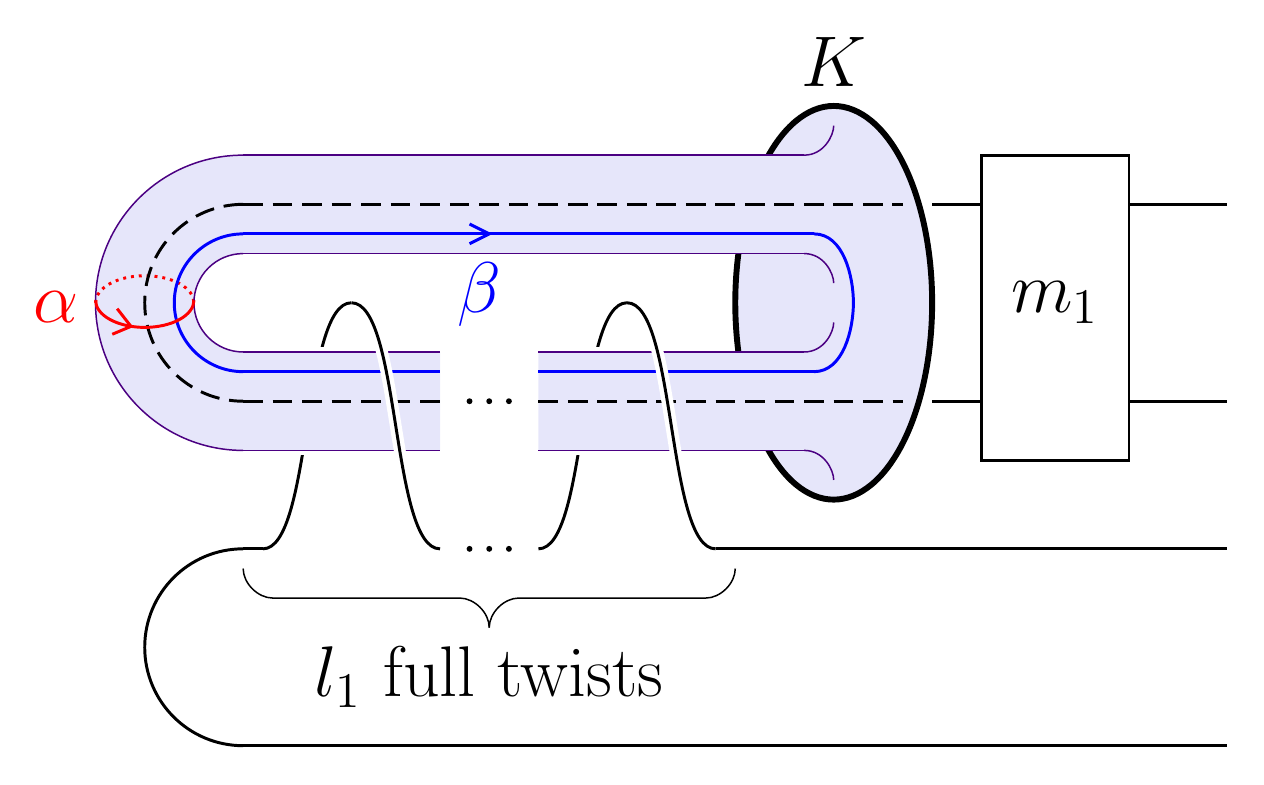}
\caption{A Seifert surface $F$ for $K$ in $\partial C_{n,k'}$.}
\label{Casson2}
\end{figure}
We regard Figure \ref{2-bridge2}-(a) as a surgery diagram of $\partial C_{n,k}$ 
and as one of $\partial C_{n,k'}$ by replacing $m_1$ with $m'_1$. 
Figure \ref{Casson1}-(a) shows a part of the diagram of $\partial C_{n,k'}$. 
A Rolfsen twist gives the diagram shown in \ref{Casson1}-(b). 
Let $K$ be the circle inserted by this change. 
The diagram in Figure \ref{Casson1}-(b) says that 
$\partial C_{n,k'}$ is obtained from $\partial C_{n,k}$ 
by $\frac{1}{d}$-surgery along $K$. 

We then compute the Alexander polynomial $\varDelta (t)$ of $K\subset \partial C_{n,k}$. 
Note that $\partial C_{n,k}$ is an integral homology sphere since $C_{n,k}$ is contractible. 
Let $F$ be a Seifert surface of $K$ shown in Figure \ref{Casson2}, 
which is homeomorphic to a once-punctured torus. 
The curves $\alpha$ and $\beta$ as shown in the figure generate $H_1(F;\Z)$, 
and the knot $K$ has the Seifert matrix 
\begin{align*}
S 
&= \left(
    \begin{array}{cc}
      lk(\alpha,\alpha^{+}) & lk(\alpha,\beta^{+}) \\
      lk(\beta,\alpha^{+}) & lk(\beta,\beta^{+}) 
    \end{array}
  \right)\\ 
&= \left(
    \begin{array}{cc}
      0 & (-1)^{n+1}\cdot l_1  \\
      (-1)^{n+1}\cdot l_1+1 & 0 
    \end{array}
  \right). 
\end{align*}
Thus we have 
\[
\varDelta (t)= \left(-l_1^2+(-1)^n\cdot l_1\right)t +2l_1^2-2(-1)^n\cdot l_1+1
+\left(-l_1^2+(-1)^n\cdot l_1\right)t^{-1}. 
\]
The surgery formula for Casson invariant gives 
\[
\lambda(\partial C_{n,k'})-\lambda(\partial C_{n,k}) 
= \frac{d}{2}\cdot \varDelta''(1)
= -d \left( l_1^2-(-1)^n\cdot l_1\right). 
\]
This value is not zero since $l_1>1$, which completes the proof. 
\end{proof}
\begin{remark}
Lemma \ref{lem: Casson} can also be shown directly by Thurston's hyperbolic Dehn filling 
for sufficiently large $k$ and $k'$. 
\end{remark}

Theorem \ref{thm: corks} follows from 
Lemmas \ref{lem: cork}, \ref{lem: upper bound}, \ref{lem: lower bound} and \ref{lem: Casson}. 
We close this section with the following remark. 
\begin{remark}
\label{rmk:primeness of C_{n,k}}
In a private discussion, 
Ruberman suggested that other examples of corks with large shadow-complexity 
can be found by performing the boundary connected-sum of corks 
whose boundary has non-zero Gromov norm. 
For instance, the corks introduced in \cite{Naoe:2015} 
whose shadow-complexity is $1$ has hyperbolic boundary. 
Let $C$ be one of them, and then the Gromov norm $\|\partial C\|$ is not zero. 
It is easily seen that the boundary connected-sum $\natural_nC$ of $n$ copies of $C$ is a cork, and 
$\displaystyle n\frac{v_{\text{tet}}}{2v_{\text{oct}}}\|\partial C\|\leq\shco(\natural_nC)\leq n$. 

There is a difference in ``primeness'' 
between the above example $\natural_nC$ and $C_{n,k}$ in Theorem \ref{thm: corks}. 
Especially, $\natural_nC$ is \textit{boundary-sum reducible}, 
and $C_{n,k}$ is \textit{boundary-sum irreducible}. 
Here an $n$-manifold $M$ with boundary is said to be \textit{boundary-sum irreducible} 
if $M_1$ or $M_2$ is homeomorphic to an $n$-ball for any decomposition $M=M_1\natural M_2$ 
\cite{Tange:2017}. 
The boundary-sum irreducibility of $C_{n,k}$ 
is shown by almost the same method as in \cite{Tange:2017} except the primeness of the boundary. 
In our case, the boundary $\partial C_{n,k}$ is a hyperbolic $3$-manifold 
due to Theorem \ref{thm:Ishikawa-Koda}, and hence it is prime. 
\end{remark}
\section{Complexity of exotic pairs}

We recall that the (special) shadow-complexity of a pair of manifolds 
is defined by the maximum between their (special) shadow-complexities. 
In this section, 
we discuss the (special) shadow-complexity of exotic pairs 
of $4$-manifolds with boundary. 
\subsection{Low complexity (nonspecial case)}
We first give the proof of Theorem \ref{thm: lowest shadow-complexity is zero} 
which answers Question \ref{question} (3). 

\begin{proof}
[Proof of Theorem \ref{thm: lowest shadow-complexity is zero}]
\begin{figure}
\includegraphics[width=125mm]{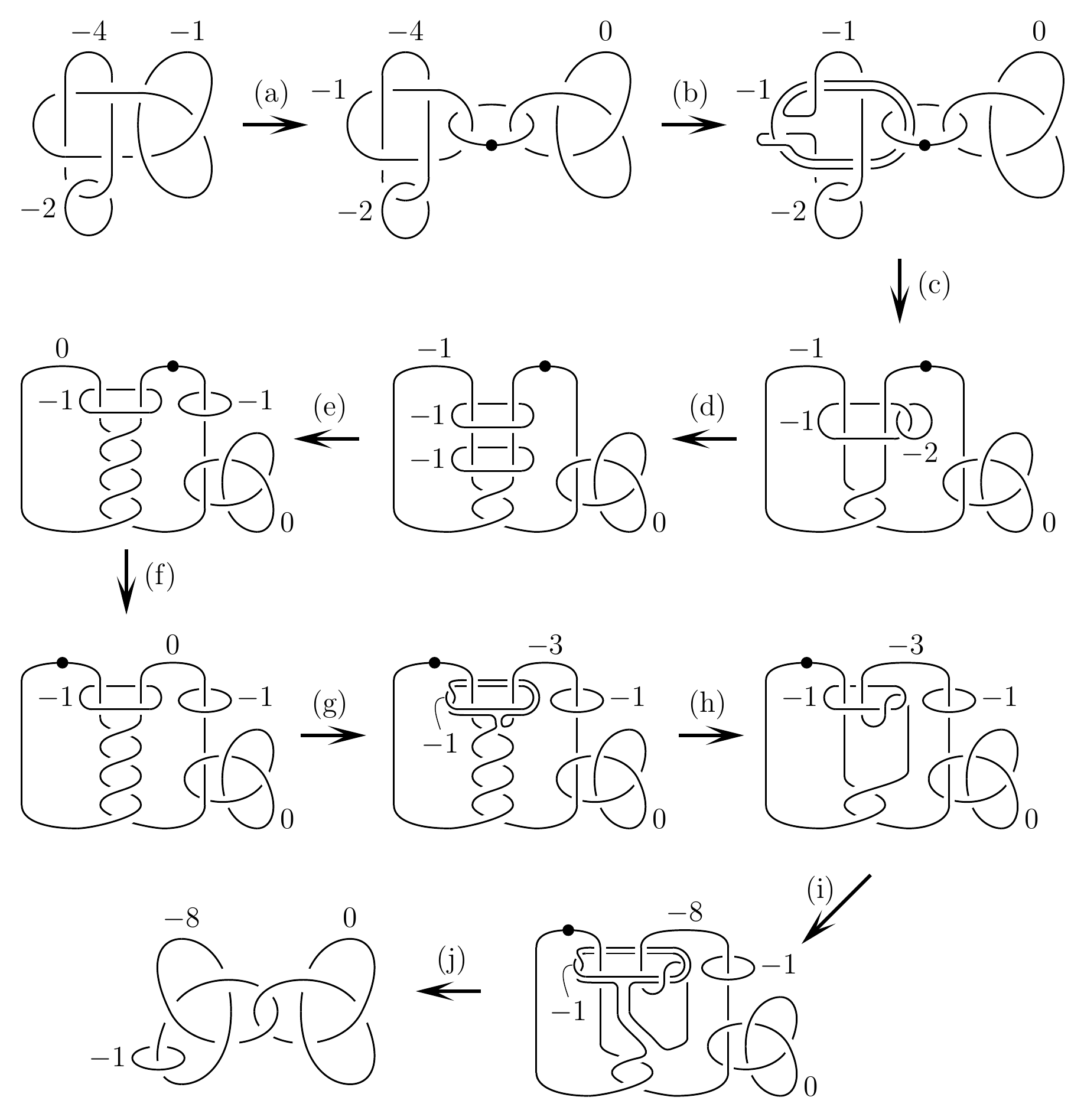}
\caption{Kirby calculus and plug twist: 
(a) create a cancelling pair. 
(b) handle slide. 
(c) isotopy. 
(d) handle slide. 
(e) handle slide. 
(f) plug twist.
(g) handle slide. 
(h) isotopy.
(i) handle slide. 
(j) delete a cancelling pair and isotopy.}
\label{plug}
\end{figure}
It is easily seen that both $W_1$ and $W_2$ have shadow-complexity zero as follows. 
For each $i=1,2$, we first project the link shown in Figure \ref{complexity zero pair} 
representing $W_i$ to a $2$-disk, 
and then glue a $2$-disk to each component of the link projection. 
The obtained polyhedron is a shadow of $W_i$. 
It collapses along a single boundary region, 
and then it becomes a shadow of $W_i$ that has no true vertices. 
Thus $\shco(W_i)=0$. 

We next check that $W_1$ and $W_2$ are exotic. 
The first diagram shown in Figure \ref{plug} represents $W_1$. 
We perform Kirby calculus in the moves (a)-(e). 
The move (f) is a \textit{plug} twist. 
A plug was introduced in \cite{Akbulut_Yasui:2008}, 
and it gives rise to many exotic pairs like a cork. 
The pair related by the move (f) is one of them 
(see \cite[Lemma 2.8 (3) for $m=1,\ n=2$]{Akbulut_Yasui:2008}). 
We perform Kirby calculus again in the moves (g)-(j), 
and we obtain the last diagram which represents $W_2$. 
Therefore $W_1$ and $W_2$ are exotic. The proof is completed. 
\end{proof}
\begin{remark}
The pair $(W_1,W_2)$ has special shadow-complexity at least $2$ since 
their second Betti number is $3$. 
Note that special polyhedra with at most 1 true vertex have 
second Betti number at most $2$. 
\end{remark}
\subsection{Low complexity (special case)}
Next we discuss Question \ref{question} (4). 
There are three homeomorphism types $X_1,\ X_2,\ X_3$ of special polyhedra 
without true vertices. 
For $i\in\{1,\ldots,6\}$, let $R_i$ be a copy of a unit disk, 
and provide a polar coordinate $(r_i,\theta_i)$ on $R_i$. 
Then the polyhedra $X_1,\ X_2,\ X_3$ are defined as 
\begin{align*}
X_1
&= R_1\ /\ (1,\theta_1)\sim(1,\theta_1+\frac{2\pi}{3}),\\ 
X_2
&= R_2\cup R_3\ /\ (1,\theta_2)\sim(1,2\theta_3),\\ 
X_3
&= R_4\cup R_5\cup R_6\ /\ (1,\theta_4)\sim(1,\theta_5)\sim(1,\theta_6). 
\end{align*}
It is easy to see that 
\[
\gl_2(R_i)=\begin{cases}
    1 &  (i=3)\\
    0 &  \text{otherwise}.
  \end{cases}
\]

Let $M_{(X_1;\gl(R_1))}$, 
$M_{(X_2;\gl(R_2),\gl(R_3))}$ and 
$M_{(X_3;\gl(R_4),\gl(R_5),\gl(R_6))}$ 
be the $4$-manifolds with boundary constructed from $X_1$, $X_2$ and $X_3$, respectively, 
by equipping with gleams $\gl(R_1),\ldots,\gl(R_6)$. 
\begin{proposition}
\label{prop: non zero scsp}
Let $l,m,n,l',m',n'$ be integers and $r,r'$ be half integers. 
\begin{enumerate}
\item
The following are equivalent: 
 \begin{itemize}
 \item[(i)] 
 $M_{(X_1;n)}$ and $M_{(X_1;n')}$ are homeomorphic, 
 \item[(ii)] 
 they are diffeomorphic,
 \item[(iii)] 
 $n=n'$.
 \end{itemize}
\item
The following are equivalent: 
 \begin{itemize}
 \item[(i)] 
 $M_{(X_2;n,r)}$ and $M_{(X_2;n',r')}$ are homeomorphic,
 \item[(ii)] 
 they are diffeomorphic,
 \item[(iii)] 
 $(n,r)=(n',r')$, or 
$n=n'\pm4$ and $r=-r'=\mp\frac{1}{2}$.
 \end{itemize}
\item
The following are equivalent: 
\begin{enumerate}
 \item[(i)] 
 $M_{(X_3;l,m,n)}$ and $M_{(X_3;l',m',n')}$ are homeomorphic, 
 \item[(ii)] 
 they are diffeomorphic,
 \item[(iii)] 
\begin{itemize}
\item 
$\{l,m,n\}=\{l',m',n'\}$, 
\item
$\{l,m,n\}=\{\pm1,\mp2,\mp2\},\ \{l',m',n'\}=\{\mp1,0,0\}$, 
\item
$\{l,m,n\}=\{0,0,\mp1\},\ \{l',m',n'\}=\{\pm1,\mp2,\mp2\}$, 
\item
$\{l,m,n\}=\{\pm1,a,b\},\ \{l',m',n'\}=\{\mp1,a\pm2,b\pm2 \}$ 
for some $a,b\in\Z$, or
\item
$\{l,m,n\}=\{1,-1,a\},\ \{l',m',n'\}=\{1,-1,a'\}$ 
for some $a,a'\in\Z$ with $a\equiv a' \pmod 2$. 
\end{itemize}
\end{enumerate}
\end{enumerate}
\end{proposition}

\begin{proof}
(1)
(iii)$\Rightarrow$(ii)$\Rightarrow$(i) is obvious. 
We assume (i) and prove (iii). 
For simplicity we write 
$M$ and $M'$ instead of $M_{(X_1;n)}$ and $M_{(X_1;n')}$, respectively.
\begin{figure}
\includegraphics[width=70mm]{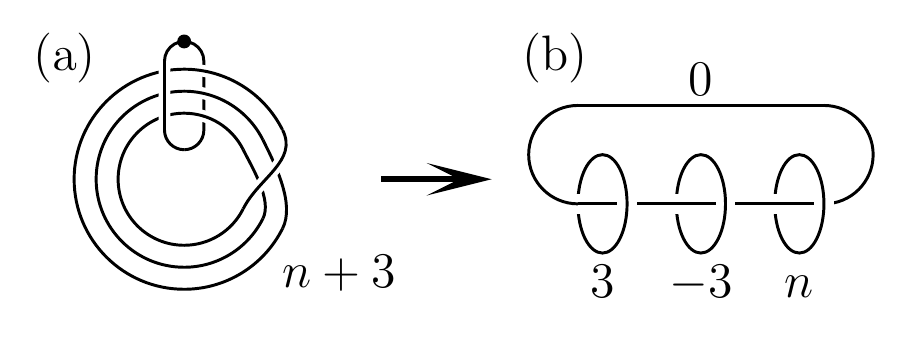}
\caption{The Kirby diagram and the surgery diagram for $M_3$ and its boundary, 
respectively.}
\label{X1}
\end{figure}
Figure \ref{X1}-(a) shows a Kirby diagram of $M$. 
By replacing $\bullet$ with $0$, 
the diagram turns to a surgery diagram of the boundary $\partial M$. 
We get the diagram pictured in Figure \ref{X1}-(b) by adding a cancelling pair.
This $3$-manifold $\partial M$
is Seifert fibered and has the form $S^2(0;(3,1),(3,-1),(n,1))$ except for $n=0$. 
This is a unique fibering except for $n\ne\pm1$. 
If $n=\pm 1$, then $\partial M \cong L(9,\mp 4)$. 
Thus if $n\ne n'$, then $M$ and $M'$ are not homeomorphic. 
Note that $\partial M \cong L(3,1)\# L(3,-1)$ for $n=0$, and it is not Seifert fibered. 

\vspace{3mm}
\noindent
(2) 
(iii)$\Rightarrow$(ii)$\Rightarrow$(i) is easy. 
We prove (i)$\Rightarrow$(iii). 
For simplicity we write $M$ and $M'$ 
instead of $M_{(X_2;n,r)}$ and $M_{(X_2;n',r')}$, respectively. 
A Kirby diagram of $M$ is pictured in Figure \ref{X2}-(a), 
\begin{figure}
\includegraphics[width=70mm]{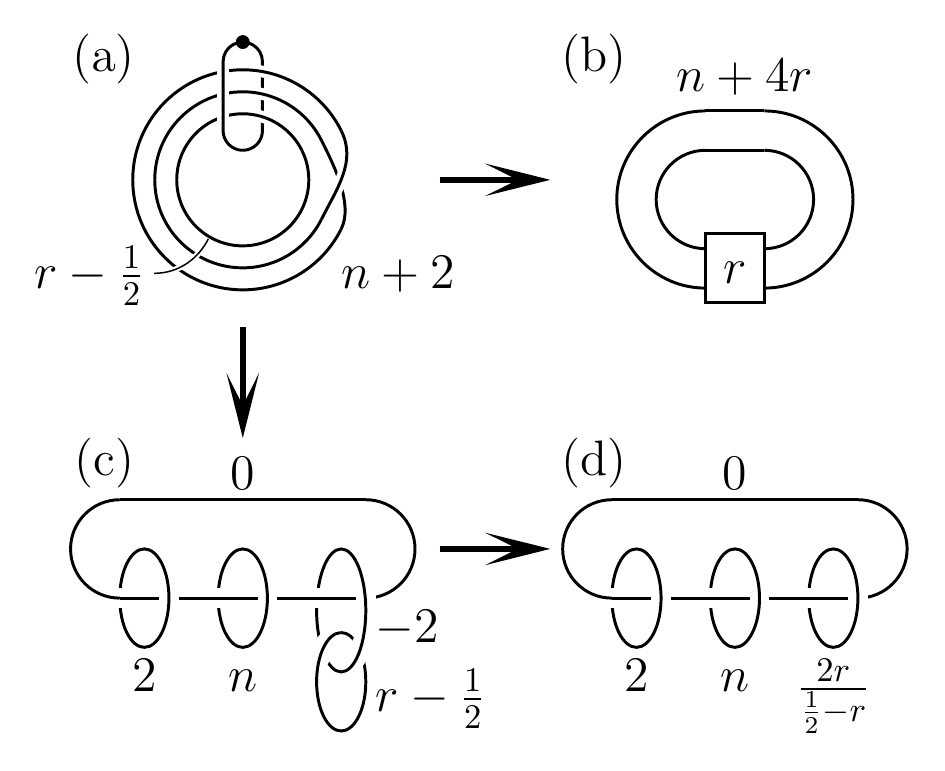}
\caption{The Kirby diagram and the surgery diagram for $M_3$ and its boundary.}
\label{X2}
\end{figure}
and we obtain the diagram of $M$ shown in Figure\ref{X2}-(b) by deleting a cancelling pair. 
Note that the box labeled $r$ represents $2r$ times half twists. 
The diagram says that the intersection form of $M$ is $\langle n+4r\rangle$. 
Smilarly, the one of $M'$ is $\langle n'+4r'\rangle$. Hence $n+4r=n'+4r'$ from (i). 
On the other hand, we obtain a surgery diagram of the boundary $\partial M$ 
shown in Figure\ref{X2}-(c) by replacing $\bullet$ with $0$ and adding a cancelling pair. 
A slum-dunk gives the diagram shown in Figure\ref{X2}-(d). 
It shows that $\partial M$ is the Seifert fibered space $S^2(0;(2,1),(n,1),(2r,\frac{1}{2}-r))$
for $n\ne0$, and so is $\partial M'$ for $n'\ne0$. 


\underline{Case $n\ne0,\pm1$ and $r\ne\pm\frac{1}{2}$} : 
The $3$-manifold $\partial M$ admits a unique Seifert fibering, and so is $\partial M'$. 
Then possible cases are $(n,2r)=(n',2r')$ or $(n,2r)=(2r',n')$. 
In the latter one, 
the intersection forms of $M$ and $M'$ are $\langle n+4r\rangle$ 
and $\langle n'+4r'\rangle = \langle 2n+2r\rangle $, respectively. 
It follows that $n=2r$, and thus $(n,2r)=(n',2r')$. 

\underline{Case $n=0,\ \pm1$ or $r=\pm\frac{1}{2}$} : In this case 
the boundary $3$-manifold $\partial M$ admits nonunique or no Seifert fiberings. 
Thus it follows that $n'=0,\ \pm1$ or $r'=\pm\frac{1}{2}$. 
The topological types of $\partial M$ is as follows:
\begin{align*}
n=0\ 
&\Longrightarrow\ 
\partial M \cong \RP^3\# L(2r,r-\frac{1}{2})\cong \RP^3\# L(2r,-2), \\
n=\pm1\ 
&\Longrightarrow\ 
\partial M\cong S^2(0;(2,1\pm2),(2r,\frac{1}{2}-r)) \cong L(4r\pm1,-4), \\
r=\pm\frac{1}{2}\ 
&\Longrightarrow\ 
\partial M\cong S^2(0;(2,\pm1),(n,1))\cong L(n\pm 2,-1).
\end{align*}
We now summarize the possible cases in Table \ref{table:complexity_0_(2)}. 
\begin{table}
\begin{center}
\begin{tabular}[c]{|c||c|c|c|c|c|}
\hline
\parbox[c][\myheighta][c]{0cm}{}
& \lower 5pt \vbox{\hbox{$n=1$}\hbox{$\wedge\ r\ne\pm\frac{1}{2}$}} 
& \lower 5pt \vbox{\hbox{$n=-1$}\hbox{$\wedge\ r\ne\pm\frac{1}{2}$}} 
& \lower 5pt \vbox{\hbox{$n=0$}\hbox{$\wedge\ r\ne\pm\frac{1}{2}$}} 
& $r=\frac{1}{2}$ & $r=-\frac{1}{2}$\\ 
\hline \hline
\parbox[c][\myheighta][c]{0cm}{}
\lower 5pt \vbox{\hbox{$n'=1$}\hbox{$\wedge\ r'\ne\pm\frac{1}{2}$}} 
& ${}^{(\star1)}r=r'$ & - & - & - & -\\
\hline
\parbox[c][\myheighta][c]{0cm}{}
\lower 5pt \vbox{\hbox{$n'=-1$}\hbox{$\wedge\ r'\ne\pm\frac{1}{2}$}} 
& ${}^{(\star2)}$ - & ${}^{(\star1)}r=r'$ & - & - & -\\
\hline
\parbox[c][\myheighta][c]{0cm}{}
\lower 5pt \vbox{\hbox{$n'=0$}\hbox{$\wedge\ r'\ne\pm\frac{1}{2}$}} 
& ${}^{(\star3)}$ - & ${}^{(\star3)}$ - & ${}^{(\star1)}r=r'$ & - & -\\
\hline
\parbox[c][\myheight][c]{0cm}{}
$r'=\frac{1}{2}$
& ${}^{(\star4)}$ - & ${}^{(\star5)}$ - & ${}^{(\star3)}$ - & ${}^{(\star1)}n=n'$ & $n=n'+4$\\
\hline
\parbox[c][\myheight][c]{0cm}{}
$r'=-\frac{1}{2}$
& ${}^{(\star4)}$ - & ${}^{(\star5)}$ - & ${}^{(\star3)}$ - & ${}^{(\star6)}n=n'-4$ & ${}^{(\star1)}n=n'$\\
\hline
\end{tabular}
\caption{}
\label{table:complexity_0_(2)}
\end{center}
\end{table}
We only show the cases $(\star1)$--$(\star5)$ in the table
because the others can be shown in much the same way as $(\star1)$--$(\star5)$. 
\begin{itemize}
\item[$(\star1)$]
It is obvious from $n+4r=n'+4r'$. 
\item[$(\star2)$]
From $n+4r=n'+4r'$, we have $r=r'-\frac{1}{2}$. 
This contradicts our assumption that both $r$ and $r'$ are half integers. 
\item[$(\star3)$]
Either $\partial M$ or $\partial M'$ is prime and the other is not. 
\item[$(\star4)$]
It is necessary that $|4r+1|=|n'\pm2|$ and 
$4\equiv 1 \pmod{|4r+1|}$
by the classification of lens spaces. 
Hence $r=\frac{1}{2}$ or $-\frac{1}{2}$, which is a contradiction. 
\item[$(\star5)$]
We have $|4r-1|=|n'\pm\frac{1}{2}|$ and $4\equiv 1\pmod{|4r-1|}$ by the classification of lens spaces. 
It follows that $r=-\frac{1}{2}$. It is a contradiction. 
\item[$(\star6)$]
It is obvious from $n+4r=n'+4r'$. 
\end{itemize}
The proof is completed.

\vspace{3mm}
\noindent
(3) 
(iii)$\Rightarrow$(ii)$\Rightarrow$(i) is easy. 
We assume (i) and prove (iii). 
To shorten notation, 
we write $M$ and $M'$ instead of $M_{(X_3;l,m,n)}$ and $M_{(X_3;l',m',n')}$, respectively.
\begin{figure}
\includegraphics[width=70mm]{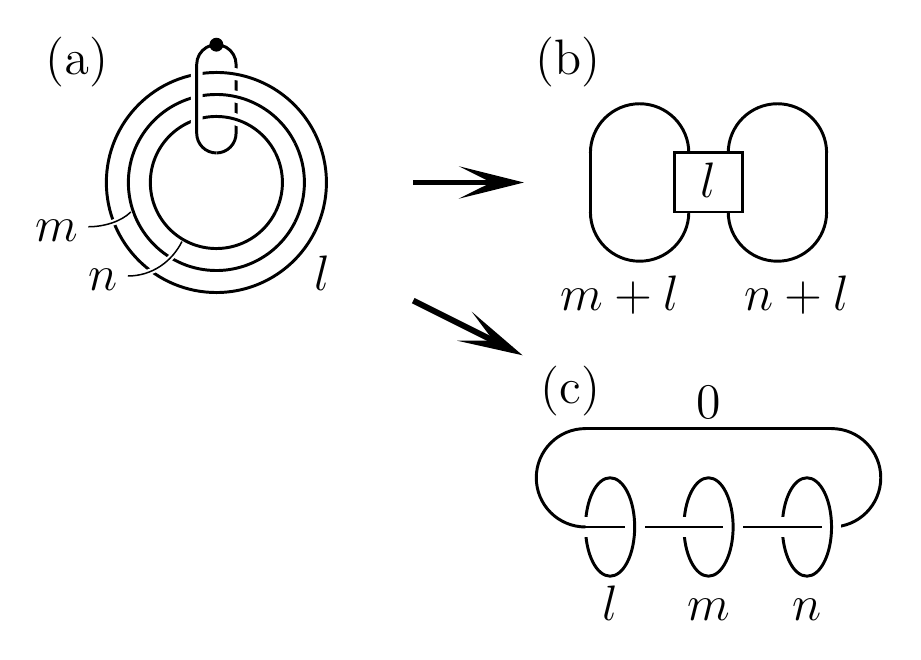}
\caption{The Kirby diagram and the surgery diagram for $M_3$ and its boundary.}
\label{X3}
\end{figure}
The $4$-manifold $M$ has the diagram shown in Figure \ref{X3}-(a).
Deleting a cancelling pair yields the diagram pictured in Figure \ref{X3}-(b). 
Thus the intersection forms of $M$ and $M'$ are 
$\displaystyle 
  \left(
    \begin{array}{cc}
      m+l & l \\
      l & n+l 
    \end{array}
  \right)$ and 
$\displaystyle 
  \left(
    \begin{array}{cc}
      m'+l' & l' \\
      l' & n'+l'
    \end{array}
  \right)$, %
respectively.
By replacing $\bullet$ with $0$, 
the diagram shown in Figure \ref{X3}-(a) turns to 
a surgery diagram of the boundary $\partial M$ shown in Figure \ref{X3}-(c). 
This $3$-manifold $\partial M$ is a Seifert fibered space of type $S^2(0;(l,1),(m,1),(n,1))$
unless $lmn=0$. 

\underline{Case where at least one of $l,m,n$ is $0$} : 
The boundary is not Seifert fibered. 
It follows that at least one of $l',m',n'$ is also $0$. 
We assume that $l=0$ and $l'=0$. 
Then the intersection forms of $M$ and $M'$ are $\langle m\rangle\oplus\langle n\rangle$ and 
$\langle m'\rangle\oplus\langle n'\rangle$. 
Thus $\{l,m,n\}$ must coincides with $\{l',m',n'\}$ by (i). 

\underline{Case where at least one of $l,m,n$ is $\pm1$ and the others are not $0$} : 
Suppose that $l=\pm1$. 
In this case, 
the boundary $\partial M$ admits a Seifert fibering with at most two exceptional fibers. 
Then so is $\partial M'$ by (i), and we may assume that $|l'|=1$. 
%

Put $\mu = m+l$, $\nu =n+l$, $\mu' = m'+l'$ and $\nu' =n'+l'$ for simplicity. 
It is easily seen that $\partial M\cong L(\mu\nu-1,\mu)$ 
and $\partial M\cong L(\mu'\nu'-1 ,\mu)$. 
The intersection form of $M$ is 
$\displaystyle 
  \left(
    \begin{array}{cc}
      \mu & 1 \\
      1 & \nu 
    \end{array}
  \right)$. 
Then it can be checked that 
\[
M\text{ is }\begin{cases}
\text{positive definite} &\hspace{-2mm} \text{if }\mu,\nu >0\text{ and }(\mu,\nu )\ne(1,1),\\
\text{negative definite} &\hspace{-2mm} \text{if }\mu,\nu <0\text{ and }(\mu,\nu )\ne(-1,-1),\\
\text{indefinite} &\hspace{-2mm} \text{otherwise}. 
  \end{cases}
\]

Now we assume that $M$ is indefinite, then so is $M'$ and we have 
\begin{itemize}
\item
$(\mu,\nu )=(1,1),(-1,-1)$, or
 \item 
$\mu\nu\leq0$,
\end{itemize}
and 
\begin{itemize}
\item
$(\mu',\nu' )=(1,1),(-1,-1)$, or
 \item 
$\mu'\nu'\leq0$. 
\end{itemize}
Assume that $\mu\nu=0$. 
This is equivalent to that $\partial M$ is homeomorphic to $S^3$ 
from $\partial M\cong L(\mu\nu-1,\mu)$. 
Thus if $\mu=0$, then we can assume $\mu'=0$. 
In this case, 
the intersection form is isomorphic to 
$\displaystyle \left(\begin{array}{cc}0 & 1 \\ 1 & 0 \end{array}\right)$ if $\nu$ is even, 
and $\langle 1\rangle\oplus\langle -1\rangle$ if $\nu$ is odd. 
Therefore $\nu\equiv\nu'\pmod{2}$.
Hence $(l,m,n)=(\pm1,\mp1,a)$ 
and $(l',m',n')=(\pm1,\mp1,a')$ or $(\mp1,\pm1,a')$
for some $a,a'\in\Z$ with $a\equiv a'\pmod{2}$. 
It follows that $\{l,m,n\}=\{1,-1,a\}$ and $\{l',m',n'\}=\{1,-1,a'\}$ 
for some $a,a'\in\Z$ with $a\equiv a'\pmod{2}$. 

Assume that $M$ is indefinite and $(\mu,\nu )=(1,1),(-1,-1)$. 
We then have $(l,m,n)=(\pm1,\mp2,\mp2)$ or $(l,m,n)=(\pm1,0,0)$. 
In either case, 
the intersection form of $M$ is isomorphic to 
$\langle 1\rangle\oplus\langle 0\rangle$ or $\langle -1\rangle\oplus\langle 0\rangle$. 
Note that $(\mu,\nu )=(1,1),(-1,-1)$ if and only if $\partial M$ is homeomorphic to $S^2\times S^1$. 
Then we have $(\mu',\nu' )=(1,1)$ or $(-1,-1)$ from $\partial M'\cong L(\mu'\nu'-1,\mu')$. 
It follows that 
\begin{itemize}
 \item 
$\{l,m,n\}=\{l',m',n'\}$, 
\item
$\{l,m,n\}=\{\pm1,\mp2,\mp2\}$ and $\{l',m',n'\}=\{\mp1,0,0\}$, or 
\item
$\{l,m,n\}=\{\mp1,0,0\}$ and $\{l',m',n'\}=\{\pm1,\mp2,\mp2\}$.
\end{itemize}

We next consider the case where $M$ is indefinite and $\mu\nu <0$. 
Then $\mu'\nu' <0$.
Since the boundaries $L(\mu\nu-1,\mu)$ and $L(\mu'\nu'-1 ,\mu')$ are homeomorphic, 
it follows that 
$\mu\nu=\mu'\nu'$, and $\mu\equiv\mu'$ or $\mu\mu'\equiv1 \pmod{-\mu\nu+1}$.
If $\mu\equiv\mu'$, then $\mu=\mu'$ since $-\mu\nu+1>|\mu|,|\mu'|>0$. 
If $\mu\mu'\equiv1$, then $\mu=\nu'$ since 
\[
\mu\mu'\equiv1 \Rightarrow \mu\mu'\nu'\equiv\nu' \Leftrightarrow \mu\equiv\nu'. 
\]
Hence $\{\mu,\nu\}=\{\mu',\nu'\}$. 
Therefore 
\begin{itemize}
 \item 
$\{l,m,n\}=\{l',m',n'\}$, or
\item
$\{l,m,n\}=\{\pm1,a,b\}$ and $\{l',m',n'\}=\{\mp1,a\pm2,b\pm2 \}$ for some $a, b\in\Z$. 
\end{itemize}

We turn to the case where $M$ and $M'$ are positive definite. 
Then we have $\mu,\nu>0$ and $\mu',\nu'>0$. 
In the same way as above, we have
\begin{itemize}
 \item 
$\{l,m,n\}=\{l',m',n'\}$, or
\item
$\{l,m,n\}=\{\pm1,a,b\}$ and $\{l',m',n'\}=\{\mp1,a\pm2,b\pm2 \}$ for some $a, b\in\Z$. 
\end{itemize}
The proof for the case of negative definite is also similar.

\underline{Case $|l|,|m|,|n|\geq2$} : 
The boundary $3$-manifold $\partial M$ has a unique Seifert fibering. 
Then we have $\{ l,m,n\}=\{ l',m',n'\}$, 
which completes the proof. 
\end{proof}
Costantino gave an upper estimate to the answer of Question \ref{question} (4) 
with manifolds shown in Figure \ref{Akbulut's pair}. 
These manifolds has been known to be exotic by Akbulut \cite{Akbulut:1991}. 
He mentioned that $M_1$ and $M_2$ have special shadow-complexity at most $1$ and $3$, respectively. 
Thus the upper estimate is $3$. 
We can easily strengthen this by Kirby calculus. 
The move (a) in Figure \ref{Akbulut's pair} is creating a cancelling pair, 
and the move (b) is an isotopy. 
Then it is immediate that $M_2$ has special shadow-complexity at most $2$. 
Combining this result and Proposition \ref{prop: non zero scsp} gives 
Theorem \ref{thm: 1 or 2}.　
\begin{figure}
\includegraphics[width=110mm]{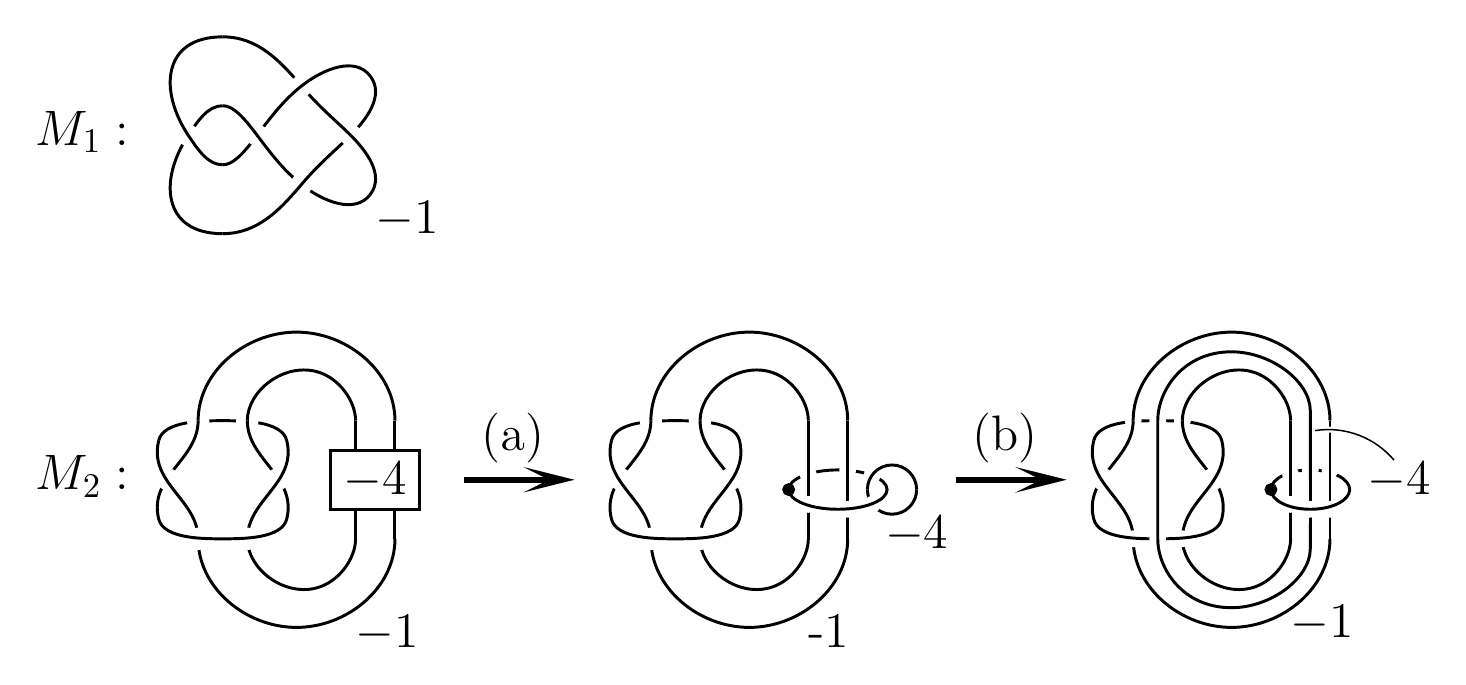}
\caption{The exotic pair having special shadow-complexity at most $2$. }
\label{Akbulut's pair}
\end{figure}

\subsection{Large complexity}
Here we show two lemmas to prove Corollary \ref{cor: exotic pair}. 
Let $W_{n,k}$ (resp. $W'_{n,k}$) be the $4$-manifold obtained from $C_{n,k}$ by attaching 
a $2$-handle with $-1$ framing along a meridian of the dotted circle 
(resp. the $0$-framed circle) in the Kirby diagram pictured in Figure \ref{Kirby diag of Cn}. 
\begin{lemma}
\label{lem: upper bound of exotic}
$\displaystyle \spshco(W_{n,k},W'_{n,k})\leq \sum_{j=1}^{n+1}l_j+n-1$.
\end{lemma}
\begin{proof}
The manifolds $W_{n,k}$ and $W'_{n,k}$ have Kirby diagrams as shown in Figures 
\ref{W1}-(a) and \ref{W2}-(a), respectively. 
\begin{figure}
\includegraphics[height=60mm]{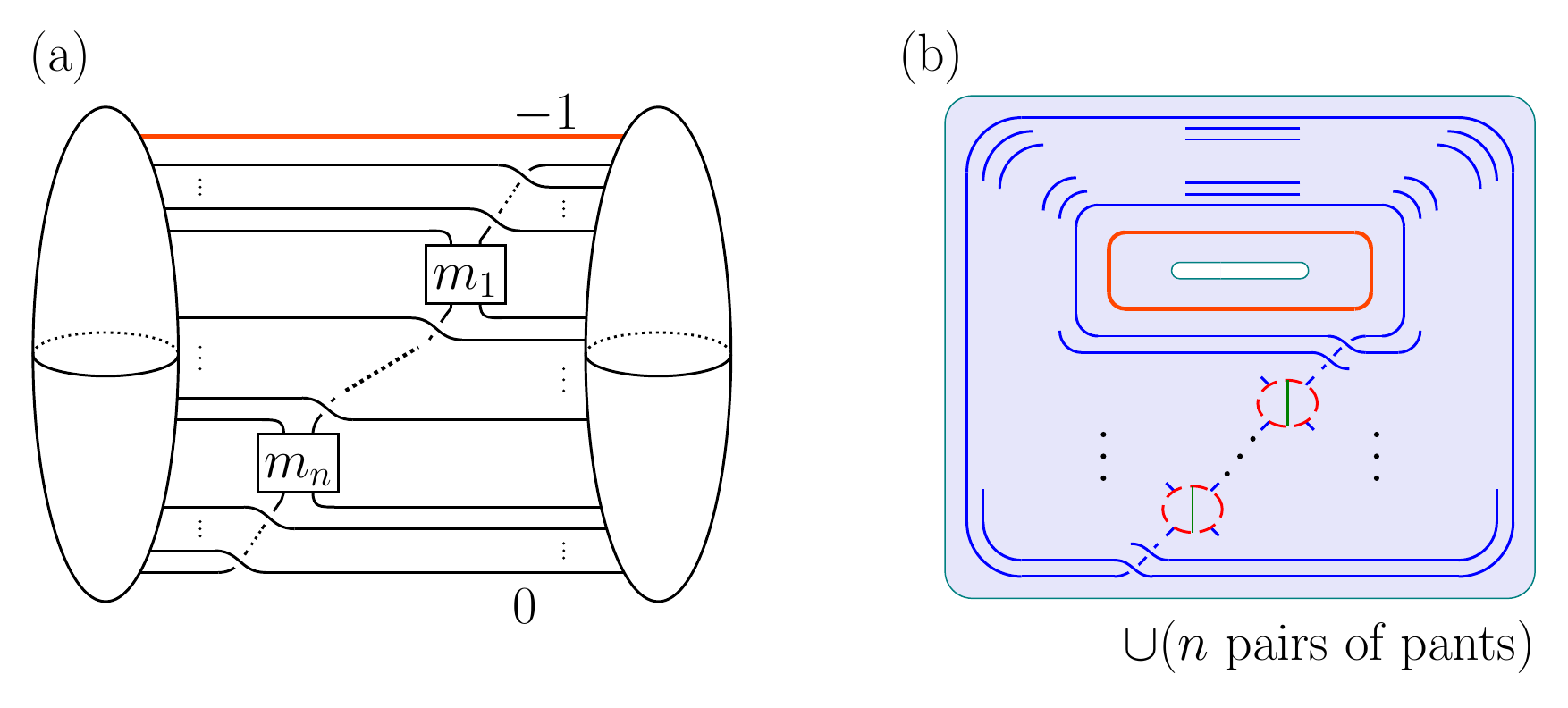}
\caption{The Kirby diagrams of $W_{n,k}$.}
\label{W1}
\includegraphics[height=50mm]{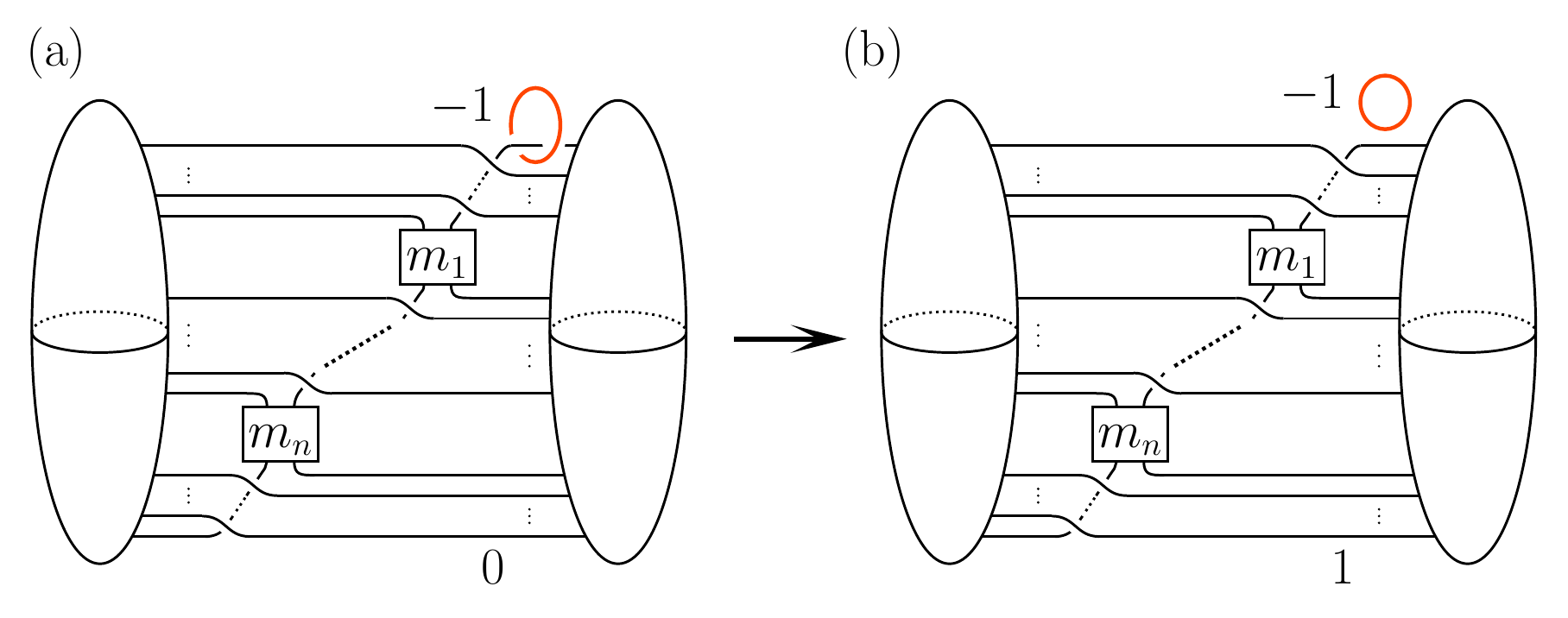}
\caption{The Kirby diagrams of $W_{n,k}$.}
\label{W2}
\end{figure}
Form the diagrams we see that each of $W_{n,k}$ and $W'_{n,k}$ admits a handle decomposition 
with a $0$-handle, $n+1$ $1$-handles and $n+2$ $2$-handles. 

We first construct a shadow of $W_{n,k}$ 
by a similar way to one in the proof of Lemma \ref{lem: upper bound}. 
The union of the $0$-handle and the $1$-handles of $W_{n,k}$ has a shadow $P$ 
considered in the proof of Lemma \ref{lem: upper bound}. 
Figure \ref{W1}-(b) shows this polyhedron $P$ with $n+2$ curves corresponding to 
the attaching circle of the $2$-handles of $W_{n,k}$. 
By attaching $n+2$ $2$-disks to these curves on this polyhedron, 
we obtain a shadow of $W_{n,k}$. 
Let us denote this polyhedron by $X$. 
It has $\displaystyle \sum_{j=1}^{n+1}(l_j-1)+6n$ true vertices, 
and $4n+1$ among them are adjacent to boundary regions. 
Since the special polyhedron obtained from $X$ by collapsing along all boundary regions of $X$ is 
also a shadow of $W_{n,k}$, 
we have $\displaystyle \spshco(W_{n,k})\leq \sum_{j=1}^{n+1}l_j+n-2$.

We next construct a shadow of $W'_{n,k}$. 
The diagram pictured in Figure \ref{W2}-(a) 
changes to the one pictured in Figure \ref{W2}-(b) by handle slide. 
Thus we have $W'_{n,k}\cong Y\# \mCP$ for some $4$-manifold $Y$. 
The manifold $Y$ admits a Kirby diagram that is same as the diagram of $C_{n,k}$ 
shown in Figure \ref{Stein position} up to framing coefficient, 
and hence it has the same shadow $P''$ as of $C_{n,k}$. 
Thus $W'_{n,k}$ has a shadow $P''$ with a bubble 
as shown in the left part of Figure \ref{bubble}. 
\begin{figure}
\includegraphics[width=90mm]{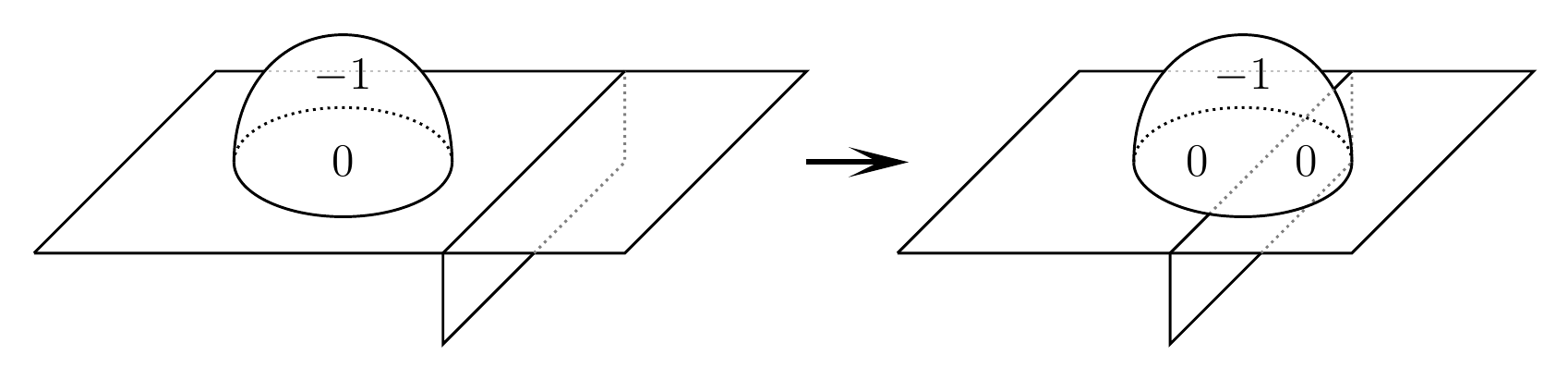}
\caption{The shadow $P''$ with a bubble and the ($0\to2$)-move. }
\label{bubble}
\end{figure}
The ($0\to2$)-move shown in Figure \ref{bubble} changes this shadow to a special one. 
The number of true vertices increases by $2$. 
It follows that  $\displaystyle \spshco(W'_{n,k})\leq \sum_{j=1}^{n+1}l_j+n-1$, 
and the lemma is proven. 
\end{proof}
\begin{lemma}
\label{lem: lower bound of exotic}
$\displaystyle 2n\leq \spshco(W_{n,k},W'_{n,k})$.
\end{lemma}
\begin{proof}
Figure \ref{W2}-(b) implies that 
the $3$-manifold $\partial W'_{n,k}$ ($\cong\partial W_{n,k}$)
is obtained by the Dehn surgery along the link shown in Figure \ref{2-bridge2}-(a) with 
coefficients $1$ and $0$. 
By the same method as in Lemma \ref{lem: lower bound}, 
this $3$-manifold also has a shadow $Q''$ with the same gleams except for 
$\displaystyle \gl(R_1)=1-\sum_{j=1}^{n+1}l_j$ 
(using the same notation as in the proof of Lemma \ref{lem: lower bound}). 
These gleams still satisfy the condition of Theorem \ref{thm:Ishikawa-Koda}, 
and thus the $3$-manifold has special shadow-complexity $2n$. 
\end{proof}
\end{document}